\documentclass[a4paper,12pt]{scrartcl}
\usepackage{amsmath,amssymb,amsthm,mathtools} 
\usepackage{psfrag}
\usepackage{epsfig}
\usepackage{color}
\usepackage{graphics}
\usepackage{multicol}
\usepackage{enumerate}
\usepackage[dvipsnames]{xcolor}

\usepackage{algorithm}
\usepackage{algorithmicx}
\usepackage{algpseudocode}
\renewcommand{\algorithmicrequire}{\textbf{Input:}}
\renewcommand{\algorithmicensure}{\textbf{Instructions:}}

\newtheorem{theorem}{Theorem}[section]
\newtheorem{lemma}{Lemma}[section]

\theoremstyle{definition}

\newtheorem{remark}{Remark}
\newtheorem{example}{Example}
\newtheorem*{localization}{Localization}

\numberwithin{equation}{section}

\newcommand{\R}{\mathbf{R}}
\def\S{\mathbf{S}}
\newcommand{\N}{\mathbf{N}}
\newcommand{\Q}{\mathbf{Q}}
\newcommand{\Z}{\mathbf{Z}}

\newcommand{\eps}{\varepsilon}
\newcommand{\oo}{\infty}
\newcommand{\nb}{\nabla}
\newcommand{\ov}{\overline}
\newcommand{\dw}{\downarrow}
\newcommand{\rw}{\rightarrow}

\newcommand{\Om}{\Omega}

\newcommand{\restr}{ \!\!\mbox{{ \Large$\llcorner$}} }

\newcommand{\be}{\begin{equation}}
\newcommand{\ee}{\end{equation}}
\DeclareMathOperator{\supp}{supp}

\DeclareMathOperator{\argmin}{argmin}

\newcommand\D{\mathcal{D}}
\newcommand\lt{\left}
\newcommand\rt{\right}
\newcommand\F{\mathcal{F}}
\newcommand\E{\mathcal{E}}

\def\s{\sigma}
\def\g{\gamma}
\def\t{\theta}
\def\a{\alpha}
\def\d{\delta}

\def\A{\mathcal{A}}
\def\M{\mathcal{M}}

\def\H{\mathcal{H}}

\def\C{\mathcal{C}}

\def\L{\mathcal{L}}
\def\LF{\mathcal{LF}}

\def\P{\mathcal{P}}
\def\e{_\eps}
\def\sm{\setminus}
\def\de{\partial}
\def\dif{\;\mathrm{d}}
\def\dx{\;\mathrm{d}x}
\def\dy{\;\mathrm{d}y}
\def\dt{\;\mathrm{d}t}
\def\dH{\;\mathrm{d}\H}

\def\txk{\tilde{x}_\eps}

\def\rwstar{\stackrel{*}{\rightharpoonup}}

\title{A simple phase-field approximation of the Steiner problem
in dimension two}
\author{A. Chambolle\and B. Merlet\and L. Ferrari}
\date{}
\begin{document}
\maketitle
\begin{abstract}
 In this paper we consider the branched transportation problem in 2D associated with a cost per unit length of the form $1+\alpha m$  where $m$ denotes the amount of transported mass and $\alpha>0$ is a fixed parameter (notice that the limit case $\alpha=0$ corresponds to the classical Steiner problem). Motivated by the numerical approximation of this  problem, we introduce a family of functionals ($\{\F_\eps\}_{\eps>0}$) which approximate the above branched transport energy. We  justify rigorously the approximation by establishing the equicoercivity and the  $\Gamma$-convergence of $\{\F_\eps\}$ as $\eps\downarrow0$. Our functionals are modeled on the Ambrosio-Tortorelli functional and are easy to optimize in practice. We present numerical evidences of the efficiency of the method. \end{abstract}
\section{Introduction}

In this paper, we introduce a phase-field approximation of a 
branched transportation energy for lines in the plane. Our main goal
is to derive a computationally tractable approximation of the
Steiner problem (of minimizing the length of lines connecting a
given set of points) in a phase-field setting. Similar results
have recently be obtained by~\cite{Bon_Lem_San}, however we believe
our approach is slightly simpler and numerically easier to implement.
We show that we can modify classical
approximations for free discontinuity problems~\cite{Mod_Mort,
Amb_Tort1,Iur,ContiFocardiIurlano} to address our specific
problem, where the limiting energy is concentrated only on a
singular one-dimensional network (and roughly measures its length). 
Numerical results illustrate the behaviour of these elliptic
approximations. In this first study, we limit ourselves
to the two-dimensional case, as in that case lines can locally
be seen as discontinuities of piecewise constant functions, so that
our construction is deriving in a quite simple ways from the above
mentioned previous works on free discontinuity problems. Higher
dimension is more challenging, from the topological point of view;
an extension of this approach is currently in preparation.

We now introduce precisely our mathematical framework.
Let $\Om\subset \R^2$ be a convex, bounded open set.
We consider measures $\s\in \M(\ov{\Om},\R^2)$ that write 
\[
\s=\t\xi\cdot\H^1\restr M,
\]
where $M$ is a $1$-dimensional rectifiable set orientated by a Borel measurable mapping $\xi:M\rw \S^1$ and  $\t:M\rw\R_+$ is a Borel measurable function representing the multiplicity. Such measure is called a rectifiable measure. We follow the notation of~\cite{OS2011} and write $\s=U(M,\t,\xi)$. Given a cost function $f\in C(\R_+,\R_+)$, we introduce the functional defined on $\M(\ov\Om,\R^2)\to\R_+\cup\{+\infty\}$ as 
\begin{equation*}
\E_f(\s):=
\begin{dcases}
\int_Mf(\t)\dH^1&\text{if }\s=U(\t,\xi,M),\\
\quad +\oo&\text{in the other cases.}
\end{dcases}
\end{equation*}
Given a sequence of $N+1$ distinct points  $S=(x_0,\dots,x_N)\in \Om^{N+1}$, we consider the minimization of $\E_f(\s)$ for $\s\in \M(\ov\Om,\R^2)$ satisfying the constraint
\begin{equation}\label{eq:divconstraint1}
\nb\cdot\s=N\d_{x_0}-\sum_{i=1}^N\d_{x_i}\quad\mbox{ in }\D'(\R^2).
\end{equation}
The distributional support of such $\sigma$  connects the source in $x_0$ to the sinks in $x_1,\cdots,x_N$. In general, a model for branched transport connecting  a set of sources to a set of sinks (represented by two discrete probabilistic measures supported on a set of points in $\Om$) is obtained by choosing $f(\t)=|\t|^\a$ with $0<\a<1$ and minimizing the associated functional under a divergence constraint similar to equation~\eqref{eq:divconstraint1}. The direct numerical optimization of the functional $\E_f$  is not easy because we do not know {\it a priori} the topological properties of the tree $M$. For this reason it is interesting to optimize an ``approximate" functional defined on more flexible objects such as functions. Such approximate model has been introduced in~\cite{OS2011} where the authors study the $\Gamma$-convergence (see~\cite{Bra2}) of a family of functionals  inspired by the well known work of Modica and Mortola~\cite{Mod_Mort}. 
Another effort in this direction can be found in the work~\cite{Bon_Lem_San} where is studied an approximation to the Steiner Minimal Tree problem~(\cite{Gilb_Poll},~\cite{Ambr_Tilli} and~\cite{Paol_Step}) by means of analogous techniques.

Here, we consider variational approximations of some energies of the form $\E_f$ through a family of functionals modeled on the Ambrosio-Tortorelli functional~\cite{Amb_Tort1}. For being more precise, we need to introduce some material. Let $\rho\in C^\infty_c(\R^2,\R_+)$ be a classical radial mollifier with $\supp \rho\subset B_1(0)$ and $\int \rho=1$. For $\eps\in (0,1]$, we set $\rho_\eps(x)=\eps^{-2}\rho(\eps^{-1}x)$ and we define the space  $V_\eps(\Om)$ of square integrable vector fields whose weak divergence satisfy the constraint
\begin{equation}\label{eq:divconstraint2}
\nb\cdot\s_\eps\, =\, \lt(N\delta_{x_0} - \sum_{j=1}^N \delta_{x_j} \rt)\ast \rho_\eps.
\end{equation}
For $\eta=\eta(\eps)>0$, we note 
\begin{equation*}\label{eq:Weps}
W_\eps(\Om)\,=\,\lt\{\phi\in H^1(\Om)\,:\, \eta \leq\phi\leq 1\mbox{ in }\Om,\ \phi\equiv 1\mbox{ on }\de\Om\rt\}.
\end{equation*}
Then we define the energy  $\F_\eps:\M(\Om,\R^2)\times L^1(\Om)\rw[0,+\oo]$ as
\begin{equation}\label{eq:functional}
\F_\eps(\s,\phi):=
  \begin{dcases}
    \;\int_{\Om} \frac{1}{2\eps}\phi^2 |\s|^2 \dx+ \int_{\Om} \frac{\eps}{2}|\nb \phi|^2  + \frac{(1-\phi)^2}{2\eps} \dx, & \text{if } (\s,\phi)\in V_\eps(\Om)\times W_\eps(\Om),\\
    \quad+\oo & \text{in the other cases.}
  \end{dcases}
\end{equation}
The first integral in the definition of the energy will be refered to as the ``constraint component'' while the second integral will be regarded as the ``Modica-Mortola component''. 
Let us briefly describe the qualitative properties of the associated minimization problem. First notice that the constraint~\eqref{eq:divconstraint2} enforces $\s$ to be non zero on a set connecting $S$. Next, the constraint component of the energy strongly penalizes $\phi^2 |\s|^2$  so that $\phi$ should be small in the region where $|\sigma|$ is large. On the other hand  the behavior of $\phi$ is controlled by the Modica Mortola component that forces $\phi$ to be close to $1$ away  from a one-dimensional set as $\eps$ converges to 0. As a consequence, we expect the support of $\sigma$ and the energy to concentrate on a one dimensional set connecting $S$. The main part of the paper consists in making rigorous and quantitative this analysis.  
\medskip

From now on, we assume that there exists some $\a\geq 0$ such that 
\begin{equation}\label{eq:limalpha}
\dfrac\eta{\eps}\ \stackrel{\eps\downarrow0}\longrightarrow\ \a.
\end{equation}  
We note $\M_S(\ov\Om)$ the set of $\R^2$-valued measures $\s\in \M(\R^2,\R^2)$ with support in $\ov{\Om}$ such that the constaint~\eqref{eq:divconstraint1} holds. We define the limit energy $\E_\a:\M(\ov \Om,\R^2)\times L^1(\Om)\rw[0,+\oo]$  as
  \begin{equation}\label{eq:limit}
   \E_\a(\s,\phi)=
   \begin{dcases}
    \;\int_{M}(1+\a\,\t)\dH^1 & \text{if } \phi\equiv1,\ \s\in \M_S(\ov\Om)\mbox{ and }\s=U(M,\t,\xi),\\
  ~  \quad+\oo &\text{in the other cases.}
   \end{dcases}
  \end{equation}
  We prove the $\Gamma$-convergence of the sequence $(\F_\eps)$ to the energy $\E_\a$ as $\eps\dw0$. 
  More precisely the convergence holds in $\M(\ov\Om,\R^2)\times L^1(\Om)$ where $\M(\ov\Om,\R^2)$ is endowed with the weak star topology and $L^1(\Om)$ is endowed with its classical strong topology.\\
    We establish the following lower bound.
\begin{theorem}[$\Gamma-\liminf$]\label{teo:sigma_liminf}
For any sequence $(\s_\eps,\phi_\eps)\subset\M(\Om,\R^2)\times L^1(\Om)$ such that $\s_\eps\rwstar\s$ and $\phi_\eps\rw\phi$ in the $L^1(\Om)$ topology, with $(\s,\phi)\in\M(\Om,\R^2)\times L^1(\Om)$
\begin{equation*}
\liminf_{k\rw+\oo}\F_\eps(\s_\eps,\phi_\eps)\geq \E_\a(\s,\phi).
\end{equation*}
\end{theorem}
 In this statement and throughout the paper, we make a small abuse of language by noting $(a_\eps)_{\eps\in(0,1]}$ and calling sequence a family $\{a_\eps\}$ labeled by a continuous parameter $\eps\in(0,1]$. In the same spirit, we call subsequence of $(a_\eps)$, any sequence $(a_{\eps_j})$ with $\eps_j\rw0$ as $j\rw+\oo$.\\
 To complete the $\Gamma$-convergence analysis, we establish the matching upper bound. 
\begin{theorem}[$\Gamma-\limsup$]\label{teo:sigma_limsup}
For any $(\s,\phi)\subset\M(\Om,\R^2)\times L^1(\Om)$ there exists a sequence $(\s_\eps,\phi_\eps)$ such that $\s_\eps\rwstar\s$ and $\phi_\eps\rw\phi$ in the $L^1(\Om)$ topology and 
\begin{equation*}
\limsup_{k\rw+\oo}\F_\eps(\s_\eps,\phi_\eps)\leq \E_\a(\s,\phi).
\end{equation*}
\end{theorem}
 Moreover, under the assumption $\a>0$ we prove the equicoercivity of the  sequence $(\F_\eps)$.
\begin{theorem}[Equicoercivity]\label{teo:sigma_equicoercive}
Assume $\a>0$. For any sequence $(\s_\eps,\phi_\eps)_{\eps\in(0,1]}\subset\M(\Om,\R^2)\times L^1(\Om)$ with uniformly bounded energies, {\it i.e.}
\[
\sup_{\eps} \F_\eps(\s_\eps,\phi_\eps)\ <\ +\oo,
\] 
there exist a subsequence $\eps_j\dw0$ and a measure $\s\in\M_S(\ov\Om,\R^2)$ such that   $\s_{\eps_j}\rightarrow \s$ with respect to the weak-$*$ convergence of measures and $\phi_{\eps_j} \rightarrow 1$ in $L^1(\Om)$. 
Moreover,  $\s$ is a rectifiable measure ({\it i.e.} it is of the form $\s=U(M,\t,\xi)$).
\end{theorem}
\begin{remark}
Observe that letting $\a=0$ in equation~\eqref{eq:limit} we obtain $\E_0(\sigma)=\H^1(\{x\in M\,:\, \t(x)>0\})$ where $\s=U(M,\t,\xi)$. This is exactly the functional associated with the Steiner Minimal Tree problem. Unfortunately, the hypothesis $\a>0$ is necessary in the compactness Theorem~\ref{teo:sigma_equicoercive}.

The fact that we are working in dimension $2$ is fundamental for the proof of Theorem~\ref{teo:sigma_liminf} as it allows to locally rewrite the vector field $\s_\eps$ as the rotated gradient of a function.
\end{remark}

\textbf{Structure of the paper:} In Section~\ref{sec:Not} we introduce some notation and several tools and notions on $SBV$ functions and vector field measures. In Section~\ref{sec:LocRes} we study a first family of energies obtained by substituting $\nb u$ for $\s$ in the definition of $\F\e$. In Section~\ref{sec:Comp} we prove the equicoercivity result, Theorem~\ref{teo:sigma_equicoercive} and we establish the lower bound stated in Theorem~\ref{teo:sigma_liminf}. In Section~\ref{sec:UppBound} we prove the upper bound of Theorem~\ref{teo:sigma_limsup}. Finally, in the last section, we present and discuss various numerical simulations.

\section{Notation and Preliminary Results}\label{sec:Not}
In the following $\Om\subset\subset\hat\Om\subset\R^d$ are bounded open convex sets. Given $X\subset\R^d$ (in practice $X=\Om$ or $X=\hat\Om)$, we denote by $\A(X)$ the class of all open subsets of $X$  and by $\A_S(X)$ the subclass of all simply connected open sets $O\subset X$ such that $\overline{O}\cap S=\varnothing$. We denote by $(e_1,\dots,e_d)$ the canonical orthonormal basis of $\R^d$, by $|\cdot|$   the euclidean norm and by $\langle\cdot,\cdot\rangle$) the euclidean scalar product in $\R^d$.  The open ball of radius $r$ centered at $x\in\R^d$  is denoted by $B_r(x)$. The $(d-1)$-dimensional Hausdorff measure in $\R^d$ is denoted by $\H^{d-1}$. We write $|E|$ to denote the Lebesgue measure of a measurable set $E\subset \R^d$. When $\mu$ is a Borel meaure and $E\subset \R^d$ is a Borel set, we denote by $\mu\restr E$ the measure defined as $\mu\restr E(F)=\mu(E\cap F)$.   \\
Let us remark that from Section~\ref{sec:Comp} onwards, we work in dimension $d=2$.\\
For  $O\in\A(\Om)$, the functional $\F_\eps(\cdot,\cdot;O)$ is the functional obtained by substuting $O$ for $\Om$ in the definition of $\F_\eps$ (see~\eqref{eq:functional}). Similarly we define the local version $\E_\a(\cdot,\cdot;O)$ of $\E_\a$.

\subsection{BV($\Om$) functions and Slicing}
BV($\Om$) is the space of functions $u\in L^1(\Om)$ having as distributional derivative $Du$ a measure with finite total variation. For $u\in BV(\Om)$, we denote by $S_u$ the complement of the Lebesgue set of $u$, that is $x\not \in S_u$ if and only if
\[
\lim_{\rho\rw0^+}\frac{1}{|B_\rho(x)|}\int_{B_\rho(x)}|u(y)-z|\dy=0
\] for some $z\in\R$. We say that $x$ is an approximate jump point of $u$ if there exist $\xi \in \S^{d-1}$ and distinct points $a,b \in \R$ such that 
\[\lim_{\rho\dw0}\frac{1}{|B^+_\rho(x,\xi)|}\int_{B^+_\rho(x,\xi)}|u(y)-a|\dy=0\qquad\text{and}\qquad\lim_{\rho\dw0}\frac{1}{|B^-_\rho(x,\xi)|}\int_{B^+_\rho(x,\xi)}|u(y)-b|\dy=0,\] 
where $B^\pm_\rho(x,\xi):=\{y\in B_\rho(x)\;:\;\pm\langle y-x,\xi\rangle\geq0\}$. Up to a permutation of $a$ and $b$ and a change of sign of $\xi$, this characterizes the triplet $(a,b,\xi)$ which is then denoted by $(u^+,u^-,\nu_u)$. The set of approximate jump points is denoted by $J_u$. The following theorem holds~\cite{Am_Fu_Pal}.
\begin{theorem}\label{teo:federer}
 The set $S_u$ is countably $\H^{d-1}$-rectifiable and $\H^{d-1}(S_u\setminus J_u)=0$. Moreover $Du\restr J_u=(u^+-u^-)\nu_u\H^{d-1}\restr J_u$ and 
 \[Tan^{d-1}(J_u,x)=\nu_u(x)^\perp\]
 for $\H^{d-1}$-a.e. $x\in J_u$.
\end{theorem}
We write the Radon-Nikodym decomposition of $Du$ as $Du= \nb u \dx+D^s u$. Setting $D^c u=D^s u \restr (\Om \setminus S_u)$ we get the decomposition 
\[Du=\nb u \dx+ (u^+-u^-)\nu_u\H^{d-1}\restr J_u+D^c u.\]
Moreover the Cantor part is such that if $\H^{d-1}(E)<\infty$, we have $|D^cu|(E)=0$ \cite[Thm.~3.92]{Am_Fu_Pal}. In particular, we have the following useful consequence
\begin{equation}
\label{BVproperty}
\H^{d-1}(E)=0\quad\Longrightarrow\quad |Du|(E)=0.
\end{equation}
We frequently use the notation $[u]$ for the jump function $(u^+-u^-):J_u\rw\R$.\\
When $d=1$ we use the symbol $u'$ in place of $\nb u$ and $u(x^\pm)$ to indicate the right and left limits of $u$ at $x$. \medskip

Let us introduce the space of special functions of bounded variation and a variant:
\[SBV(\Om):=\{u\in BV(\Om):D^cu=0\},\]
\[GSBV(\Om):=\{u\in L^1(\Om):\max(-T,\min(u,T))\in SBV(\Om)\;\forall T\in\R\}.\]
Eventually, in Section~\ref{sec:LocRes}, the following space of piecewise constant functions will be useful. 
\[\P(\Om)=\{u\in GSBV(\Om):\nb u=0\}.\]
To conclude this section we recall the slicing method for functions with bounded variation. Let $\xi\in\S^{d-1}$ and let 
\[\Pi_\xi:=\{y\in\R^d:\langle y,\xi\rangle=0\}.\]
If $y\in\Pi_\xi$ and $E\in\R^d$, we define the one dimensional slice 
\[E_{\xi,y}:=\{t\in\R: y+t\xi\in E\}.\]
For $u:\Om\rw\R$, we define $u_{\xi,y}:\Om_{\xi,y}\rw\R$ as
\begin{equation*}
u_{\xi,y}(t):=u(y+t\xi), \quad t\in \Om_{\xi,y}.
\end{equation*}
Functions in $GSBV(\Om)$ can be characterized by one-dimensional slices (see \cite[Thm.~4.1]{Bra1})
\begin{theorem}\label{teo:braides}
 Let $u\in GSBV(\Om)$. Then for all $\xi \in\S^{d-1}$ we have 
 \[u_{\xi,y}\in GSBV(\Om_{\xi,y})\quad \text{for } \H^{d-1}\text{-a.e. }y\in\Pi_\xi.\]
Moreover for such $y$, we have 
\begin{gather*}
 u'_{\xi,y}(t)=\langle\nb u(y+t\xi),\xi\rangle\quad \text{for a.e. } t\in\Om_{\xi,y},\\
 J_{u_{\xi,y}}=\{t\in \R:y+t\xi\in J_u\},
\end{gather*}
and
\begin{equation*}
 u_{\xi,y}(t^\pm)=u^\pm(y+t\xi)\quad \text{or}\quad u_{\xi,y}(t^\pm)=u^\mp(y+t\xi)
\end{equation*}
according to whether $\langle\nu_u,\xi\rangle>0$ or  $\langle\nu_u,\xi\rangle<0$. Finally, for every Borel function $g:\Om\rw \R$,
\begin{equation}\label{eq:braidesricostruzione}
 \int_{\Pi_\xi}\sum_{t\in J_{u_{\xi,y}}} g_{\xi,y}(t)\dH^{d-1}(y)=\int_{J_u}g|\langle\nu_u,\xi\rangle|\dH^{d-1}.
 \end{equation}
Conversely if $u\in L^1(\Om)$ and if for all $\xi \in\{e_1,\dots,e_d\}$ and almost every $y\in\Pi_\xi$ we have $u_{\xi,y}\in SBV(\Om_{\xi,y})$ and 
\[\int_{\Pi_\xi}|Du_{\xi,y}|(\Om_{\xi,y})\dH^{d-1}(y)<+\oo\]
then $u\in SBV(\Om)$.
\end{theorem}

\subsection{Rectifiable vector Measures}
Let us introduce the linear operator $\perp$ that associates to each vector $v=(v_1,v_2)\in \R^2$ the vector $v^\perp=(-v_2,v_1)$ obtained via a $90^\circ$ counterclockwise rotation of $v$. Notice that the $\perp$ operator maps divergence free $\R^2$-valued measures onto curl free $\R^2$-valued measures. Let $O\subset\R^2$ be a simply connected and bounded  open set. By Stokes Theorem, for any divergence free measure $\s\in\M(O,\R^2)$ there exists a function $u\in BV(\Om)$ with zero mean value such that $\s= Du^\perp$. On the other hand for $u\in \P(\Om)$ $\sigma:=Du^\perp$ is divergence free and by Theorem~\ref{teo:federer}, $\sigma =(u^+-u^-)\nu_u^\perp \H^1=U(J_u,[u],\nu_u^\perp)$.\smallskip

Let us now produce an elementary example of measure $\g$ of the form $U(M,\t,\xi)$.
\begin{example}\label{ex:div2punti}
   Given two points $x,y\in\Om$ we consider the smooth path from $x$ to $y$ defined as
  \[r(t):=x+t\frac{y-x}{|x-y|}\quad \mbox{ for } t\in [0,|x-y|].\]
  We define the measure $\g\in \M(\Om,\R^2)$  by
  \[
  (\phi,\g)\,:=\,\int\langle\phi(r(t)),\dot{r}(t)\rangle\dt\qquad\mbox{for any } \phi\in \C(\overline{\Om},\R^2).
  \]
We then have $\g=U([x,y],1,\xi)$ with $\xi=\frac{y-x}{|y-x|}$. Notice that $\nb\cdot\g=\d_{x}-\d_y$.\\
Similarly we obtain a measure satisfying this property by substituting for $r$ any Lipschitz path from $x$ to $y$.
\end{example}
We will make use of the following construction.
\begin{lemma}\label{lemma:completediv}
 Given $S=(x_0,\cdots,x_N)\in \Om^{N+1}$ a sequence of $N+1$ distinct points, there exist a vector measure $\g= U(M_\g,\t_\g,\xi_\g)$ and a finite partition  $(\Om_i)\subset \A(\Om)$ of $\Om$ such that 
   \begin{enumerate}[\quad\quad a)]
   	\item $\nb\cdot\g=-N\d_{x_0}+\sum_{i=1}^{N} \d_{x_i}$,
   	\item $\t_\g:M_\g\rw\{1,N\}$, 
	\item each $\Om_i$ is a polyhedron,
   	\item $ M_\g\subset \bigcup_i \de\Om_i$,
   	\item $\Om_i$ is of finite perimeter for each $i$ and $\Om_i\cap\Om_j=\emptyset$ for $i\neq j$, 
   	\item $\L^2(\Om\sm\cup_i\Om_i)=0$.
\end{enumerate}	
Moreover if $M$ is a $1$ dimensional countably rectifiable set, we can choose $\gamma$ and $(\Om_i)$ such that  $\H^1(M\cap \bigcup_i \de\Om_i)=0$.

 \end{lemma}	
\begin{proof}
Let us fix a point $p\in\Om\setminus S$. By Example~\ref{ex:div2punti} we can construct a measure $\g_i$ with $\nb\cdot\g_i=\d_{x_i}-\d_{p}$ for $i\in\{0,\cdots,N\}$. We define \[\g=-N \g_0 +\sum_{i=1}^N\g_i.\]
By construction $(a)$ holds true. Moreover, up to a small displacement of $p$ we may assume that $[p,x_i]\cap[p,x_j]=\{p\}$ for $i\neq j$ so that $(b)$ holds. \\
Next, let $D_j$ be the straight line supporting $[p,x_j]$. We define the sets $(\Om_i)$ as the connected components of
$\Om\setminus \left(D_0 \cup\cdots\cup D_N\right)$. We see that $(c,d,e,f)$ hold true.\\
 For the last statement, we observe that  by the coarea formula,  we have $\H^1(M\cap \bigcup_i \de\Om_i)=0$ for a.e. choice of $p$. 
\end{proof}
\begin{figure}[h!]
\centering
  \includegraphics[width=.4\textwidth]{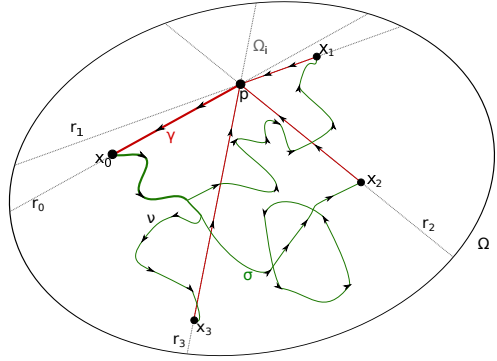} 
  \caption{Example of the construction of the $\H^1$-rectifiable measure $\g$ (red) and of the partition $\{\Om_i\}$ (gray) in the case $\s$ (green) is being an $\H^1$-rectifiable vector measure.}\label{fig:partizione}
\end{figure}

\section{Local Result}\label{sec:LocRes}
In this section we introduce  a localization of the family of functionals $(\F_\eps)$ (see~\eqref{eq:functional}). We establish  a lower bound and a compactness property for these local energies.

\begin{localization}\label{procedure}
Let $O\in A_S(\Om)$, for $u\eps\in H^1(O)$ and $\phi_\eps\in H^1(O)$, we define 
\[
\LF_\eps(u_\eps,\phi_\eps;O)\ :=\ \F_\eps(\nb u_\eps,\phi_\eps;O).
\]
Notice that for $\eps<d(O,S)$, we have  $\nb\cdot \s_\eps\equiv 0$  in $O$ for any $\s_\eps\in V_\eps(\Om)$. By the Stokes theorem we have $Du_\eps=\s_\eps^\perp$ for some  $u\in H^1(O)$ and we have
\begin{equation*}
\F_\eps(\s_\eps,\phi_\eps;O)=\LF_\eps(u_\eps,\phi_\eps;O).
\end{equation*}
\end{localization}
The remaining of the section is devoted to the proof of 
\begin{theorem}\label{teo:localResult}
Let $(u_\eps)_{\eps\in(0,1]}\subset H^1(O)$ be a family of functions with zero mean value and let $(\phi_\eps)\subset H^1(O)$ such that $\phi_\eps\in H^1(O,[\eta(\eps),1])$. Assume  that $c_0:=\sup_\eps \LF_\eps(u_\eps,\phi_\eps;O)$ is finite,  then there exist a subsequence $\eps_j$ and a function $u\in BV(\Om)$ such that
\begin{enumerate}[a)]
 \item $\phi_{\eps_j} \rw 1$ in $L^2(O)$,
 \item $u_{\eps_j}\rw u$ with respect to the weak-$*$ convergence in $BV$,
 \item $u\in \P(O)$.
\end{enumerate}
Furthermore for any $u\in\P(O)$ and any sequence $(u_\eps,\phi_\eps)$ as above such that $u_\eps\rwstar u$, we have the following lower bound of the energy:
\[\liminf_{\eps\rw0}\LF_\eps(u_\eps,\phi_\eps;O)\geq \int_{J_u\cap O}\lt[1+\alpha|[u]|\rt]\dH^{d-1}.\]
\end{theorem}
The proof is achieved in several steps and mostly follows ideas from \cite{Iur}
(see also~\cite{ContiFocardiIurlano}).
In the first step we obtain~\textit{(a)} and~\textit{(b)}. In steps 2. we prove~\textit{(c)} and the lower bound for one dimensional slices. Finally in step 3. we prove~\textit{(c)} and the lower bound in dimension $d$. The construction of a recovery sequence that would complete the $\Gamma$-limit analysis is postponed to the global model in Section~\ref{sec:UppBound}.
\begin{proof}
 \textit{Step 1.} Item \textit{(a)} is a straightforward consequence of the definition of the functional. Indeed, we have 
 \begin{equation*}
\int_O(1-\phi_\eps)^2\, \dx\, \leq \eps\, \LF_\eps(\s_\eps,\phi_\eps)\, \leq\, c_0\, \eps\,  \stackrel{\eps\dw0}\longrightarrow \ 0.
\end{equation*}
For~\textit{(b)}, since $(u_\eps)$ has zero mean value, we only need to show that $ \sup_\eps\{|Du_\eps|(O): k\in\N\}<+\infty$. Using Cauchy-Schwarz inequality we get
\begin{equation} 
\label{eq:compactness1}
\lt[|Du_\eps|(O)\rt]^2=\lt(\int_{O}|\nb u_\eps|\rt)^2\leq \lt(\eps\int_{O}\frac{1}{\phi^2\e}\rt)\lt(\frac{1}{\eps}\int_{O}\phi^2\e|\nb u_\eps|^2\rt).
\end{equation}
By assumption, the second therm in the right hand side of \eqref{eq:compactness1} is bounded by $2c_0$. In order to estimate the first term we split $O$ in the two sets $\{\phi_\eps<1/2\}$ and $\{\phi_\eps\geq1/2\}$. We have,\begin{equation*}
\eps\int_{O}\frac{1}{\phi\e}\,=\, \eps\int_{\{\phi_\eps<1/2\}}\frac{1}{\phi^2\e}+\eps\int_{\{\phi\e\geq1/2\}}\frac{1}{\phi^2\e}.
\end{equation*}
Since $\phi_\eps\geq \eta$ and $(1-t)^2$ is decreasing in the interval $O$ the following inequalities hold
\begin{eqnarray*}
\int_{\{\phi\e<1/2\}}\frac{1}{\phi^2\e}&\leq& \frac{2 \eps}{\eta^2(1-1/2)^2}\int_{O}\frac{(1-\phi\e)^2}{2\eps}\,\leq\,\frac {8\eps}\eta^2 c_0,\\
\nonumber
\int_{\{\phi_\eps\geq1/2\}}\frac{1}{\phi^2\e}&\leq &\int_{O}\frac{1}{(1/2)^2} \, =\, 4|O|.
\end{eqnarray*}
Combining these estimates with~\eqref{eq:compactness1} we obtain
\begin{equation}
\label{eq:limitatezzagradu}
\lt[|Du_\eps|(O)\rt]^2\, \leq\, \frac{\eps^2}{\eta^2}\:16c_0^2 + 8\eps |O| c_0\,\stackrel{\eps\dw0}\longrightarrow\,  \frac{16c_0^2}{\a^2}\,<\,\infty.
\end{equation}
This establishes~\textit{(b)}.\medskip

\textit{Step 2.} In this step we suppose $O$ to be an interval of $\R$. We first prove that $u$ is piecewise constant.  The idea is that in view of the constraint component of the energy, variations of $u_\eps$ are balanced by low values of $\phi_\eps$. On the other hand the Modica-Mortola component of the energy implies that $\phi_\eps\simeq 1$ in most of the domain and that transitions from $\phi_\eps\simeq 1$ to $\phi_\eps\simeq 0$ have a constant positive cost (and therefore can occur only finitely many times).\medskip

\textit{Step 2.1. proof of $u\in \P(O)$.} Let us define 
\begin{equation}\label{def:insiemi}
B_\eps := \lt\{x\in O:\phi_\eps(x)<\frac{3}{4}\rt\}\,\supset\,  A_\eps := \lt\{x\in O:\phi_\eps(x)<\frac{1}{2}\rt\},
\end{equation}
and let
\begin{equation}\label{def:C}
C_\eps = \{I \text{ connected component of } B_\eps : I\cap A_\eps\neq \varnothing \}.
\end{equation}
Let us show that the cardinality of $C_\eps$ is bounded by a constant independent of $\eps$. Let $\eps$ be fixed and consider an interval $I\in C_\eps$. Let $a,b\in \bar I$ such that $\{\phi_\eps(a),\phi_\eps(b)\}=\{1/2,3/4\}$. Using the usual Modica-Mortola trick, we have 
\begin{equation*}
 \LF_\eps(u_\eps,\phi_\eps;I)\,\geq\,\int_I \eps|\phi'_\eps|^2  + \frac{(1-\phi_\eps)^2}{4\eps} \dx \geq \int_{(a,b)} |\phi'_\eps|(1-\phi_\eps)\dx\,\geq\, \int_{1/2}^{3/4} (1-t)\,dt \, =\, \dfrac3{2^5}.
\end{equation*}
Since all the elements of $C_\eps$ are disjoint, we deduce from the energy bound that  
\[
\# C_\eps\leq2^5c_0/3,
\]
where we note $\# C_\eps$ the cardinality of $C_\eps$. Next, up to extraction we can assume that  $\# C_\eps=N$ is fixed. The elements of $ C_\eps$ are written on the form $I^\eps_i=(m^\eps_i-w_i^\eps,m^\eps_i+w_i^\eps)$ for $i=1,\cdots,N$ with $m^\eps_i<m^\eps_{i+1}$. Since $\phi_\eps\rw1$ in $L^1(O)$ we have
\begin{equation}\label{misura:zero}
\sum_{I^\eps_i\in C_\eps}|I^\eps_i|= \sum_i2w^\eps_i \rw 0.
\end{equation}
Up to further extraction, we can assume that each sequence $(m_i^\eps)$ converges in $\ov O$. We  call $m_1\leq m_2 \leq \dots \leq m_N$ their limits. We now prove that $|D u|(O\setminus \{m_i\}_{i=1}^N) =0$. For this, we fix  $x\in O\setminus\{m_i\}_{i=0}^N$ and establish the existence of a neighborhood $B_\delta(x)$ of $x$ for which $|Du|(B_\delta(x))=0$. Let $0<\delta\leq 1/2\min_i|x-{m}_i|$. Equation~\eqref{misura:zero} ensures that for $\eps$ small enough $B_\delta(x)\cap C_\eps=\varnothing$. 
Notice that from the definitions in~\eqref{def:insiemi} and~\eqref{def:C} we have that $\phi_\eps\geq 1/2$ outside $C_\eps$. Hence, using Cauchy-Scwarz inequality, we have for $\eps$ small enough, 
\begin{equation*}
\lt(\int_{B_\delta(x)}| u'_\eps|\dx\rt)^2\leq 2\delta \int_{B_\delta(x)}| u'_\eps|^2\dx\leq (2\delta)(2\eps) 4\lt(\frac{1}{2\eps}\int_{B_\delta(x)}\phi_\eps|^2 u'_\eps|^2\dx\rt)\, \leq\, 16c_0\eps\delta \,\stackrel{\eps\dw0}\longrightarrow\,0.
\end{equation*}
By lower semicontinuity of the total variation on open sets we conclude that $|Du|(B_\delta(x))=0$. This establishes $u\in \P(O)$ with $J_u\subset \{m_1,\cdots,m_N\}$.\medskip

\textit{Step 2.2. Proof of the lower bound.} Without loss of generality, we prove the lower bound in the case $J_u=\{0\}$ and $ D:=u(0^+)=-u(0^-)>0$. Using the same argument as in~\cite[Pag.~7]{Iur} for any $0<d<D$ there exist six points 
$y_1<x^1_\eps\leq\tilde{x}^1_\eps<\tilde{x}^2_\eps\leq x^2_\eps<y_2$ such that 
\begin{gather*}
\lim_{\eps \rw 0}\phi_\eps(y_1)= \lim_{\eps \rw 0} \phi_\eps(y_2) = 1,\\
\lim_{\eps \rw 0}\phi_\eps(x^1_\eps)= \lim_{\eps \rw 0} \phi_\eps(x^2_\eps) =0,\\
u_\eps(\tilde{x}^1_\eps)= -D+d,\quad\quad \; u_\eps(\tilde{x}^2_\eps)= D-d.
\end{gather*}
Using the Modica-Mortola trick in the intervals $(y_1,x^1_\eps)$ and $(x^2_\eps,y_2)$ as above, we compute:
\begin{equation}\label{dis1:liminf} 
\liminf_{\eps \dw 0}\LF_\eps(u_\eps,\phi_\eps;(y_1,x^1_\eps)\cup(x^2_\eps,y_2))\geq\liminf_{\eps \dw 0}\int^{x^1_\eps}_{y_1}(1-\phi_\eps)|\phi'_\eps| \dx+\int_{x^2_\eps}^{y_2}(1-\phi_\eps)|\phi'_\eps| \dx\,\geq\,1.
\end{equation}
For the estimation on the interval $I_\eps=(\txk^1,\txk^2)$ let us introduce:
\begin{gather*}
G_\eps := \lt\{ w\in H^1(I_\eps) : w(\txk^1)=-D+d,\, w(\txk^2)=D-d \rt\},\\
Z_\eps := \lt\{ z\in H^1(I_\eps) :\eta\leq z\leq1 \text{ a.e. on } I_\eps \rt\},\\
H_\eps(w,z):=\int_{I_\eps}\lt(\frac{1}{2\eps}z^2|w'|^2+\frac{(1-z)^2}{2\eps}\rt)\dx,\\
h_\eps(z)= \inf_{w\in W_\eps}H_\eps(w,z) \text{ for } z\in Z_\eps.
\end{gather*}
Note that by Cauchy-Schwarz inequality, we have for $w\in G_\eps$ and $z\in Z_\eps$,
\[
\int_{I_\eps}z^2|w'|^2\, \geq \, \left(\int_{I_\eps}|w'|\dx\right)^2\lt(\int_{I_\eps}\frac{1}{z^2}\rt)^{-1}\, \geq\, 4(D-d)^2\lt(\int_{I_\eps}\frac{1}{z^2}\rt)^{-1}.
\]
We deduce the lower bound
\begin{equation}\label{eq:minimo}
h_\eps(z) \, \geq  \, 4(D-d)^2 \lt(2\eps\int_{I_\eps}\frac{1}{z^2}\dx\rt)^{-1}+ \int_{I_\eps}\lt(\frac{(1-z)^2}{2\eps}\rt)\dx.
\end{equation}
Let us remark that optimizing $H_\eps(w,z)$ with respect to $w\in G_\eps$ we see that this inequality is actually an equality.\\
Consider for $0 <\lambda<1$ the inequalities:
\begin{equation*}
\int_{\{x\in I_\eps: \phi_\eps\geq \lambda\}}\frac{1}{\phi^2_\eps}\leq \frac{\L^1(I_\eps)}{\lambda^2}\qquad\mbox{ and }\qquad\int_{\{x\in I_\eps: \phi_\eps< \lambda\}}\frac{1}{\phi^2_\eps}\leq \frac{1}{(1-\lambda)^2}\frac{2\eps}{\eta^2}\lt(\int_{I_\eps}\frac{(1-\phi_\eps)^2}{2\eps}\dx\rt).
\end{equation*}
Applying both of them in~\eqref{eq:minimo} we obtain
\begin{align}
\nonumber
\LF_\eps(u_\eps,\phi_\eps,I_\eps)&\geq  h_\eps(\phi_\eps)\\ \nonumber
&\geq\frac{2(D-d)^2}{\frac{\eps\L^1(I_\eps)}{\lambda^2}+\frac{1}{(1-\lambda)^2}\frac{2\eps^2}{\eta^2}\lt(\int_{I\e}\frac{(1-\phi\e)^2}{2\eps}\dx\rt)}+ \int_{I\e}\lt(\frac{(1-\phi\e)^2}{2\eps}\rt)\dx \\
&\geq 2(1-\lambda)\frac{\eta}{\eps}(D-d)-(1-\lambda)^2\frac{\eta^2}{2\eps}\frac{\L^1(I\e)}{\lambda^2} \label{dis2:liminf}
\end{align}
where the latter is obtained by minimizing the function:
\[
t \mapsto \frac{2(D-d)^2}{\frac{\eps\L^1(I_\eps)}{\lambda^2}+\frac{1}{(1-\lambda)^2}\frac{2\eps^2}{\eta^2}t}+ t.
\]
Therefore we can pass to the limit in~\eqref{dis2:liminf} and obtain:
\begin{equation*}
 \liminf_{\eps\dw0}\LF_\eps(u_\eps,\phi_\eps,I_\eps)\geq(1-\lambda)\a\,2(D-d).
\end{equation*}
Sending $\lambda$ and $d$ to $0$ and recalling the estimation in~\eqref{dis1:liminf} we get
\begin{equation}\label{eq:liminf1d}
\liminf_{\eps \dw 0}\LF_\eps(u_\eps,\phi_\eps,(y_1,y_2))\, \geq\, 1+\a\,2D \, =\, 1+\a|u(0^+)-u(0^-)|.
\end{equation}
\textit{Step 3.} Using Fubini's decomposition we can rewrite the energy obtaining for every $\xi\in \S^{d-1}$, 
\begin{equation*}
 \LF_\eps(u_\eps,\phi_\eps;O)\,\geq\,\int_{\Pi_\xi}\lt( \int_{O^\xi_y}\frac{1}{2\eps}(\phi^2_\eps)^\xi_y |(u'_\eps)^\xi_y|^2  + \frac{\eps}{2}|(\phi'_\eps)^\xi_y |^2  + \frac{(1-(\phi_\eps)^\xi_y )^2}{2\eps} \dt\rt)\dH^{d-1}(y).
\end{equation*}
Since $\LF_\eps(u_\eps,\phi_\eps;O)$ is bounded, for $\H^{d-1}$ almost every $y\in \Pi_\xi$, the inner integral in the latter is also bounded, furthermore it corresponds to the functional on the one dimensional slice $O^\xi_y$ studied in the previous step evaluated on the couple $((u_\eps)^\xi_y,(\phi_\eps)^\xi_y)$. Taking the $\liminf$ of the above quantity, using~\eqref{eq:liminf1d}, we get by Fatou's lemma that for any $\xi\in \S^{d-1}$ and $\H^{d-1}$ almost every $y\in \Om_\xi$
\begin{equation*}
\int_{\Pi_\xi}\sum_{m_i\in(J_{u})^\xi_y}\lt[1+\a|u^\xi_y(m_i^+)-u^\xi_y(m_i^-)| \rt]\mathrm{d}\H^{d-1}(y)\leq \liminf_{\eps\dw0} \LF_\eps(u_\eps,\phi_\eps;O).
\end{equation*}
Therefore in force of Theorem~\ref{teo:braides} we have $u\in SBV(O)$. Moreover, since $(u')^\xi_y=0$ on each slice, we have  $u\in \P(O)$. Applying identity~\eqref{eq:braidesricostruzione} we get 
\begin{equation}\label{eq:liminfxi} 
\liminf_{\eps\rw0}\LF_\eps(u_\eps, \phi_\eps;O) \geq \int_{J_u\cap O}|\nu_u\cdot \xi|\lt[1+\a|[u]|\rt]\dH^{d-1}.
\end{equation}
In order to conclude, we use the following localization method stated by Braides in~\cite[Prop.~1.16]{Bra1}.
\begin{lemma}\label{lem:braides}
Let $\mu:\A(X)\rw[0,+\oo)$ be a superadditive set function and let $\lambda$ be a positive measure on $X$. For any $i\in\N$ let $\psi_i$ be a Borel function on $X$ such that $\mu(A)\geq\int_A \psi_i\dif \lambda$ for all $A\in\A(X)$. Then 
\[\mu(A)\geq\int_A\psi\dif\lambda\]
where $\psi:=\sup_i\psi_i$.
\end{lemma}
\noindent
We introduce the superadditive increasing set function $\mu$ defined on $\A(O)$ by 
\[\mu(A):=\Gamma-\liminf_{\eps\rw0} \LF_\eps(u;)),\quad\text{for any }A\in\A(O)\]
and we let $\lambda$ be a Radon measure defined as 
\[\lambda:=[1+\a|u(x^+)-u(x^-)|]\H^{d-1}\restr J_u.\]
Fix a sequence $(\xi_i)_{i\in\N}$ dense in $\S^{d-1}$. By~\eqref{eq:liminfxi} we have 
\[\mu(O)\geq \int_O \psi_i\dif \lambda, \quad i\in \N,\]
where
\[\psi_i(x):=
\begin{dcases}
   |\langle\nu_u(x),\xi_i\rangle|&\text{if } x\in J_u,\\
   0&\text{if } x\in O\setminus J_u.
  \end{dcases}
\]
Hence by Lemma~\ref{lem:braides} we finally obtain
\[\liminf_{\eps\rw0}\LF_\eps(u_\eps,\phi_\eps;O)\geq \int_O \sup_i\psi_i(x)\dif \mu =\int_{J_u\cap O}\lt[1+\a|[u]|\rt]\dH^{d-1}.\]
\end{proof}

\section{Equicoercivity and $\Gamma$-liminf}
\label{sec:Comp}
We first prove the compactness property stated in the introduction. Let us consider a sequence $(\s_\eps,\phi_\eps)\in\M(\Om,\R^2)$ uniformly bounded in energy by $c_0<+\oo$, 
\begin{equation}\label{eq:unifbound}
0\leq\F_\eps(\s_\eps,\phi_\eps)\leq c_0\qquad\mbox{for } \eps\in(0,1].
\end{equation}
\begin{proof}[Proof of Theorem~\ref{teo:sigma_equicoercive}]
First observe that by definition~\eqref{eq:functional} and equation~\eqref{eq:unifbound}, we have $\s_\eps\in V_\eps(\Om)$ and $\phi_\eps\in W_\eps(\Om)$.

Next,  substituting $|\s_\eps|$ for  $|\nb u_\eps|$  in the argument of~\textit{Step 1.} of the proof of Theorem~\ref{teo:localResult}, inequality~\eqref{eq:limitatezzagradu}  reads
\begin{equation*}
 |\s_\eps|(\Om)\leq\sqrt{16\frac{\eps^2}{\eta^2}\:c_0^2 + 8\eps |\Om| c_0} \,\stackrel{\eps\dw0}\longrightarrow\,  \frac{4 c_0}{\a}\,<\,\infty.
\end{equation*}
Thus the total variation of $(\s_\eps)$ is uniformly bounded and there exists $\s\in \M_S(\ov{\Om})$ such that up to extraction $\s_\eps\to\s$ weakly-$*$ in $\M(\ov\Om)$.

Now, considering the last term in the energy~\eqref{eq:functional} we have
\begin{equation*}
\int_\Om(1-\phi_\eps)^2 \dx\leq 2\eps \;\F_\eps(\s_\eps,\phi_\eps)\leq 2\eps\;c_0 \rightarrow 0.
\end{equation*} 
Hence, $\phi_\eps\to1$ in $L^2(\Om)$.\medskip

Let us now study the structure of the limit measure $\s$. Let us recall that   $\hat\Om$ is a bounded convex open set such that $\ov\Om\subset\hat\Om$ and let us extend $\s_\eps$ by 0 and $\phi_\eps$ by 1 in $\hat\Om\setminus\ov\Om$. Obviously we have $\F_\eps(\s_\eps,\phi_\eps;\hat\Om)=\F_\eps(\s_\eps,\phi_\eps;\Om)$, therefore for any $O\in\A_s(\hat\Om)$ applying the localization described in Section~\ref{sec:LocRes} we can associate to each $\s_\eps$ a function $ u_\eps\in H^1(O)$ with mean value $0$ such that $\s_\eps=\nb^\perp u_\eps$ in $O$. By Theorem~\ref{teo:localResult} there exists $u\in\P(O)$ such that up to extraction $u_\eps\rwstar u$. Eventually, by uniqueness of the limit, we get 
 \[\s\restr O=-[u]\nu_{J_u}^\perp\H^1\restr (J_u\cap O).\]
Since we can cover $\ov\Om\setminus S$ by finitely many sets $O\in \A_s(\hat\Om)$, this shows that $\s$ decomposes as 
\begin{equation*}
\s\, =\, U(M_\s,\t_\s,\xi_\s)+\underbrace{\sum_{j=0}^N c_j \delta_{x_j}}_{\mu}. 
\end{equation*}
By Lemma~\ref{lemma:completediv} there exists a rectifiable measure $\g=U(M_\g,\t_\g,\xi_\g)$ such that $\nb\cdot(\s+\g)=0$ and $\H^1(M_\g\cap M_\s )=0$. Then there exists $u\in BV(\Om)$ such that $Du =\s^{\perp}+\g^{\perp}$. From~\eqref{BVproperty}, we deduce $|Du|(S)=0$ which implies $|\mu|(S)=\sum|c_j|=0$. Hence $c_j=0$ for $j=0,\dots,N$ and $\s$ writes in the form $U(M_\s,\t_\s,\xi_\s)$.
\end{proof}

Let us now use the local results of Section~\ref{sec:LocRes}  to prove the lower bound.
\begin{proof}[Proof of Theorem~\ref{teo:sigma_liminf}]
Let $(\s_{\eps},\phi_{\eps})$ as in the statement of the theorem. Without loss of generality, we can suppose that $\F_{\eps}(\s_{\eps},\phi_{\eps})<+\oo$. Theorem~\ref{teo:sigma_equicoercive} then ensures the existence of a rectifiable measure $\s=U(M_\s,\t_\s,\xi_\s)$ with $\supp(\s)\subset\overline\Om$ such that $\s_{\eps}\rwstar\s$. \\ 
Let $\hat\Om$ be as in the previous proof and let us define $\mu=\Gamma-\liminf_{\eps}\F_{\eps}(\s_{\eps},\phi_{\eps})$ and $\lambda= \a|\s|+\H^{1}\restr M_\s$. Consider the countable family of sets $\{O_{i}\}\subset\A_{S}(\hat\Om)$ made of the open rectangles $O_i\subset \hat\Om \setminus S$ with vertices in $\Q^{2}$ and let $\psi_{i}:=1_{O_{i}}$.
The local result stated in Theorem~\ref{teo:localResult} gives for any $i\in\N$
\[\mu(A)\geq\mu(O_{i}\cap A)\geq\lambda(O_{i}\cap A)=\int_{A}\psi_{i}\dif \lambda.\]
Therefore Lemma~\ref{lem:braides} gives
\begin{align*}
\Gamma-\liminf_{\eps\dw0}\F_{\eps}(\s_{\eps},\phi_{\eps})&=\mu(\hat\Om)\geq\lambda(\hat\Om)=\a|\s|(\overline\Om)+\H^{1}(M_\s) 
\end{align*}
since $\sup_{i}\psi_{i}$ is the constant function $1$.
\end{proof}

\section{Upper bound}\label{sec:UppBound}
\subsection{A density result}
In order to obtain the upper bound we first provide a density lemma. We show that measures which have support contained in a finite union of segments, are dense in energy.\\
Without loss of generality let us assume that $\s\in \M_S(\ov\Om)$  is such that $\E_\alpha(\s)<\infty$. In particular $\s=U(M_\s,\t_\s,\xi_\s)$ is a $\H^1$-rectifiable measure. Applying Lemma~\ref{lemma:completediv} we obtain a $\H^1$-rectifiable measure $\g=U(M_\g,\t_\g,\xi_\g)$ and a partition of $\Om$  made of polyhedrons $\{\Om_i\}$  such that $M_\g\subset \cup_i\de\Om_i$, $\H^1(M_\s \cap\cup_i\de\Om_i)=0$ and $\s+\g$ is divergence free. \\
From the above properties, we can write
\[
\s^\perp+\g^\perp=Du
\]
 for some $u\in\P(\Om)$. Our strategy is the following, using existing results~\cite{Bel_Cham_Gold}, we build an  approximating sequence for $u$ on each $\Om_j$ whose gradient is supported on a finite union of segments. We then glue these approximations together to obtain a sequence $(w_j)$ approximating $u$ in $\hat\Om$. The main difficulty is to establish that $Dw_j \restr [\cup_i\de\Om_i]$ is close to $Du\restr  [\cup_i\de\Om_i]=\g^\perp$.

\begin{lemma}[Approximation of $u$]\label{lem:uapproximation}
  There exists a sequence $(w_j)\subset\P(\hat\Om)$ with the following properties:
  \begin{enumerate}[\quad a)]
  \item $w_j\to u$ weakly in $BV(\hat\Om)$,
  \item $\supp w_j\subset \ov\Om$,
  \item $\limsup_{j\rw\oo}\E_\a(w_j,1)\leq\E_\a(u,1)$,
  \item $J_{w_j}$ is contained in a finite union of segments for any $j\in \N$,
  \item $|Dw_j-Du|(\cup \de\Om_i)\to0$.
  \end{enumerate}
\end{lemma}
\begin{proof}
\textit{Step 1.} 
In order to apply the results of~\cite{Bel_Cham_Gold}, we first need to modify $u$ and the energy.  Let us note the energy density function $f(t)=1+\a t$ and for $k\geq0$ and $t\geq 0$ let us introduce the approximation 
  \begin{equation*}\label{eq:definizionefk}
 f_k(t):=\min\{(2^{k/2}+\a 2^{-k/2})\sqrt{t},f(t)\}.
  \end{equation*}
  \begin{figure}[h!]
    \centering
    \includegraphics[width=.3\textwidth]{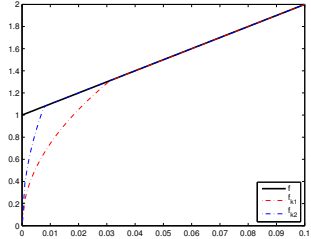} 
\caption{Graph of f and two of its approximations $f_{k_1} $ and $f_{k_2}$ with $k_1<k_2$.}\label{fig:approssimazionef}
  \end{figure}
 We have $0\leq f_k\leq f$ and $f_k\equiv f$ on $[2^{-k},+\infty)$.
Notice that $f_k$ is continuous, sub-additive and increasing on $[0,+\oo)$ and that  $f_k(0)=0$ with $\lim_{t\rw0} \frac{f_k(t)}{t}=+\oo$. 
We define the associated energy for functions $v\in\P(\hat\Om)$ as $\E_{f_k}(v,\hat \Om):=\int_{J_{v}\cap\hat\Om}f_k([v])\dH^1$. \\
Now we note $\P_k(\hat\Om)$ the set of functions $v\in \P(\hat\Om)$ such that $v(\hat\Om)\subset 2^{-k}\Z$. For these functions we have $|v^+(x)-v^-(x)|\geq 2^{-k}$ for $\H^1$-almost every $x\in J_v$. Consequently, there holds
\[
\E_{f_k}(v)\, =\, \E_f(v).
\]
For each fixed $k\geq0 $,  let us introduce the function 
\[
u_k=2^{-k}\lfloor2^k u\rfloor
\]
where $\lfloor t\rfloor$ denotes the integer part of the real $t$. Note that $u_k\in \P_k(\hat\Om)$ with $J_{u_k}\subset J_u$ and $\|u-u_k\|_\infty\leq 2^{-k}$. Notice also since $\lt|(u_k^+-u_k^-)-(u^+-u^-)\rt|\leq 2^{-k}$ we have 
\begin{equation}
\label{eq:diffuuk}
|Du_k-Du|(\hat\Om)\, \leq\, 2^{-k}\H^1(J_u)
\end{equation}
In particular $u_k\to u$ strongly in $BV(\hat\Om)$. Moreover, we see that
\begin{equation}\label{eq:Ef=Efk}
\E_{f_k}(u_k) \, =\, \E_f(u_k) \, \leq\ \, \E_f(u) + \a 2^{-k}\H^1(J_u).
\end{equation}

\medskip\textit{Step 2.} Let us approximate the function $u_k$. Let us fix $k\geq 0$ and $\Om_i$. We can apply Lemma 4.1 of~\cite{Bel_Cham_Gold} to  the function $u_k\restr{\Om_i}$ and to the energy $\E_{f_k}(\cdot,\Om_i)$. We obtain a sequence $(w^i_j)$ which enjoys the following properties:
  \begin{equation*}\label{eq:properties}
  \begin{aligned}
        & w^i_j(\Om_i)\subset u_k(\Om_i)\subset 2^{-k}\Z,\quad \forall j\in\N, \mbox{ hence }w^i_j\in \P_k(\hat\Om), \\
        & w^j_i\to u_k \mbox{ in } L^1(\Om_i)\mbox{ as }j\rw +\oo, \\
        & \lim_{j\rw +\oo} \E_{f_k}(w^i_j,\Om_i)= \lim_{j\rw +\oo} \E_{f}(w^i_j,\Om_i)=\E_{f}(u_k,\Om_i),\\  
        & J_{w^j_i} \mbox{ is contained in a finite union of segments for any } j\in \N,\\
        & \int_{\de \Om_i}|Tw^i_j-Tu_k|\dH^1\rw0\mbox{ where } T:BV(\Om_i)\rw L^1(\de \Om_i) \mbox{ denotes the trace operator}.\\
     \end{aligned}
  \end{equation*}
Let us now define globally 
\[
 w_j:=\sum_j w_j^i 1_{\Om_i}.
\]
From the above properties, we have  $w_j\rwstar u_k$, 
\begin{equation}
\label{eq:convEfw}
\lim\E_{f}(w^i_j,\hat \Om)=\E_{f}(u_k,\Om_i)
\end{equation}
and 
\begin{equation}
\label{eq:diffukwj}
|D w_j - D u_k|(\cup_i\de\Om_i)\, \to\, 0\quad\mbox{as $j\to\infty$}. 
\end{equation}
Eventually, using a diagonal argument, we have proved the existence of a sequence $(w_j)\subset \P(\hat\Om)$ complying to items~\textit{(a)},~~\textit{(b)} and~\textit{(d)} of the lemma. Moreover, item~\textit{(c)} is the consequence of~\eqref{eq:Ef=Efk} and~\eqref{eq:convEfw} and  item~\textit{(e)} follows from~\eqref{eq:diffuuk} and~\eqref{eq:diffukwj}.
\end{proof}

Going back to the $\H^1$-rectifiable measures $\s=U(M_\s,\t_\s,\xi_\s)$, we define the sequence 
\[
 \s_j:=-Dw^\perp_i-\g.
\]
We recall that $\g=U(M_\g,\t_\g,\xi_\g)$ with $M_\g\subset \cup \de \Om_i$. In particular $\g=-Du^\perp\restr(\cup_i \de\Om_i)$. We deduce from the previous lemma:
\begin{lemma}\label{lem:density}
 There exists a sequence $(\s_j)\in\M_S(\ov\Om)$ with the properties:
  \begin{enumerate}[\quad-]
  \item $\s_j\rw \s$ with respect to weak-$*$ convergence of measures,
  \item $\s_j= U(M_{\s_j},\t_{\s_j},\xi_{\s_j})$ with $M_{\s_j}$ contained in a finite union of segments,
  \item $\limsup_{j\rw\oo}\E_\a(\s_j,1)\leq\E_\a(\s,1)$.
    \end{enumerate}
\end{lemma}

 \subsection{Construction of a recovery sequence}
 Let us prove the $\Gamma$-limsup inequality stated in Theorem~\ref{teo:sigma_limsup}. Recall that the latter consists in finding a sequence $(\s\e,\phi\e)$ for any given couple $(\s,\phi)\in\M(\overline\Om,\R^2)\times L^1(\Om)$ such that $\s\e\rwstar\s$,  $\phi\e\rw\phi$ in $L^1(\Om)$ and 
 \begin{equation}\label{eq:limsup}
\limsup_{\eps\dw0}\F\e(\s\e,\phi\e)\leq\E_\a(\s,\phi).  
 \end{equation}
 When $\E_\a(\s,\phi)=+\infty$ the inequality is valid for any sequence therefore by definition~\eqref{eq:limit} we can assume $\s=U(M,\t,\xi)$ and $\phi=1$. Furthermore by density Lemma~\ref{lem:density} is sufficient to consider measures of the form
 \begin{equation}\label{eq:strutturadense}
\s=\sum_{i=1}^n U(M_i,\t_i,\xi_i), 
 \end{equation}
 where $M_i$ is a segment, $\t_i\in\R_+$ is $\H^1$-a.e. constant and $\xi_i$ is an orientation of $M_i$ for each $i$. Without loss of generality we can suppose that for each couple of segments $M_i$, $M_j$, for $i \neq j$, the intersection $M_i\cap M_j$ is at most a point (called branching point) not belonging to the relative interior of $M_i$ and $M_j$. 
 We firstly produce the estimate~\eqref{eq:limsup} for $\s$ composed by a single segment thus let us assume $\s=\t e_1\cdot \H^1\restr (0,l)\times\{0\}$.  
 
 \medskip\noindent\textit{Notation:}
 Let us fix the values 
\begin{equation*}\label{eq:infinitesimi}
a_\eps:=\begin{dcases}
         \frac{\t\a\,\eps}{2} &\mbox{ if } \a>0\\
         \quad\eps& \mbox{ if } \a=0
        \end{dcases},
\qquad\qquad
b_\eps:=\eps\ln\lt(\frac{1-\eta}{\eps}\rt)
\qquad\mbox{ and }\qquad
r\e=\max\{\eps,a\e\}.
\end{equation*}
Let  $d_\infty(x,S)$ be the distance function from $x$ to the set $S\subset \Om$ relative to the infinity norm on $\R^2$ and $Q_r(P)=\{x\in\R^2: d_\infty(x,P)\leq r\}$ the square centered in $P$ of size $2r$ and sides parallel to the axes. Introduce the sets
\begin{align*}\label{def:rectangular}
I_{a\e}&:=\{x\in\R^2: d_\infty(x,[0,l]\times\{0\})\leq a\e\}\cup Q_{r\e}(0,0)\cup Q_{r\e}(l,0),\\
I_{b\e}&:=\{x\in\R^2: d_\infty(x,I_{a\e})\leq b\e\},\\
I_{c\e}&:=\{x\in\R^2: d_\infty(x,(I_{a\e}\cup I_{b\e}))\leq \eps\},\\
I_{d\e}&:=\Om\setminus(I_{a\e}\cup I_{b\e}\cup I_{c\e}),
\end{align*}
and define $R\e= I_{a\e}\setminus(Q_{r\e}(0,0)\cup Q_{r\e}(l,0))$. 
\begin{figure}[h!]
\centering
  \includegraphics[width=\textwidth]{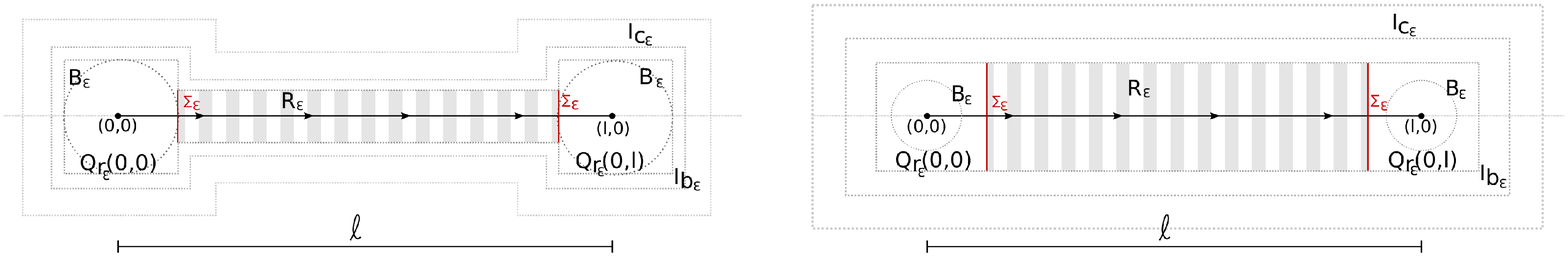} 
  \caption{Example of the neighborhoods of the segment $[0,l]\times\{0\}$. On the left the case $r\e=\eps$ on the right the case in which $r\e=a\e>\eps$. The stripped region is $R\e$ and $I_{a\e}=R\e\cup(Q_{r\e}(0,0)\cup Q_{r\e}(l,0))$. Remark that $\supp(\rho\e)=B(0,\eps)$.}\label{fig:intorninuovo}
\end{figure}

\medskip\noindent\textit{Costruction of $\s\e$:} We build $\s\e$ as a vector field supported on $I_{a\e}$. In particular we add together three different constructions performed respectively on $R\e$, $Q_{r\e}(0,0)$ and $ Q_{r\e}(l,0)$. 
Let $r=r_\eps/\eps$ and consider the problem 

\begin{minipage}[b]{0.6\textwidth}
\begin{equation*}
\begin{dcases}
\Delta u=\pm\t\delta_{x_0}*\rho &\mbox{ on }  Q_{r}(0,0),\\
\frac{\de u}{\de \nu}=\frac{\pm\t}{\H^1(\Sigma)}  &\mbox{ on }  \Sigma^\pm=\{x\in\R^2: x_1=\pm1,\;|x_2|\leq\frac{\t\a}{2}\}.
\end{dcases}
\end{equation*}
\end{minipage}
\hfill
\begin{minipage}{0.3\textwidth}
    \begin{center}
    \includegraphics[width=.6\textwidth]{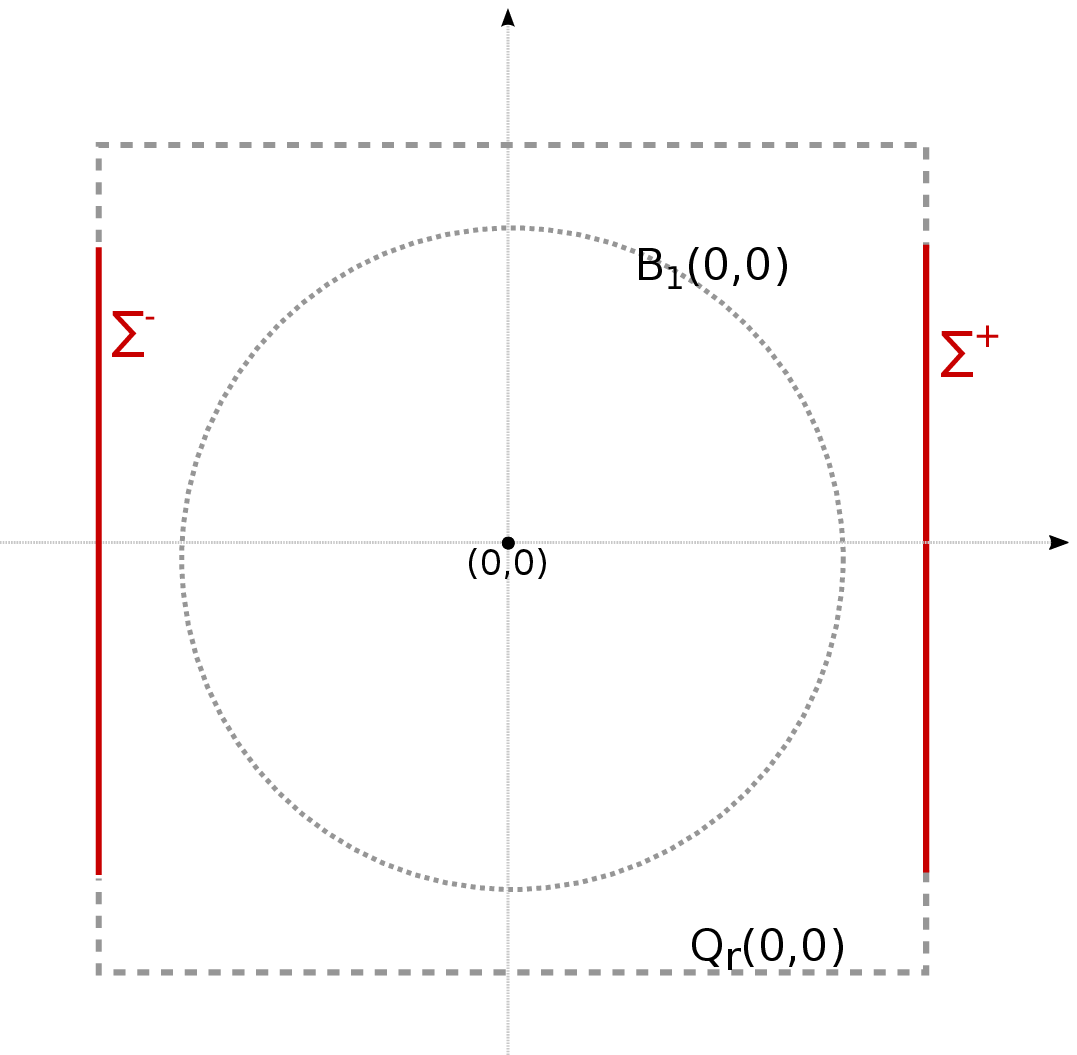} 

    \end{center}
\end{minipage}

\medskip\noindent Let $u^+$ be the solution relative to the problem in which every occurrence of $\pm$ is replaced by $+$ and let $u^-$ defined accordingly. Then set 
\begin{equation}\label{def:se}
 \s\e=
 \begin{dcases}
       \frac{\nb u^+(x/\eps)}{\eps} &\mbox{ on } Q_{r\e}(0,0),\\
       \quad\frac{\t}{2a_\eps}\cdot e_1 &\mbox{ on } R\e,\\
       \frac{\nb u^-((x-(l,0))/\eps)}{\eps} &\mbox{ on } Q_{r\e}(l,0).
  \end{dcases}
\end{equation}
By construction we have that $\nb\cdot\s\e=\t(\d_{(0,0)}-\d_{(l,0)})\ast\rho\e$ and $\s\e\rwstar\s$.
Let us point out as well that there exists a constant $c(\a,\t)$ such that
\begin{equation}\label{eq:energiapunti}
c(\a,\t):=\int_{Q_{r\e}(l,0)}|\s\e|^2\dx=\int_{Q_{r\e}(0,0)}|\s\e|^2\dx=\int_{Q_{r}(0,0)}\lt|\nb u^+(x)\rt|^2\dx=\int_{Q_{r}(0,0)}\lt|\nb u^-(x)\rt|^2\dx.
\end{equation}

\medskip\noindent\textit{Costruction of $\phi\e$:}
Most of the properties of $\phi\e$ are a consequence of the inequalities obtained in Theorem~\ref{teo:localResult} and the structure of $\s\e$. On one hand we need $\phi\e$ to attain the lowest value possible on $I_{a\e}$ in order to compensate the concentration of $\s\e$ in this set, on the other, as shown in inequality \eqref{dis1:liminf}, we need to provide the optimal profile for the transition from this low value to $1$. For this reasons we are led to consider the following ordinary differential equation associated with the optimal transition
\begin{equation}\label{eq:ode}
\begin{dcases}
w'_\eps = \frac{1}{\eps}(1-w_\eps), \\
w_\eps(0)= \eta. \\
\end{dcases}
\end{equation}
Observe that $w_\eps=1-(1-\eta)\exp\lt(\frac{-t}{\eps}\rt)$ is the explicit solution of equation~\eqref{eq:ode} and set 
\begin{equation}\label{def:phie}
  \phi_\eps(x):=\begin{dcases}
  		\eta&\mbox{ if } x\in I_{a_\eps},\\
		   w_\eps(d_\infty(x, I_{a\e})) &\mbox{ if } x\in I_{b_\eps},\\
		 d_\infty(x, I_{b\e})-\eps+1&\mbox{ if } x\in I_{c_\eps},\\
       		1&\mbox{ otherwise.}
                \end{dcases}
\end{equation}
 The choice of the behavior in the region $I_{c\e}$ is given by the fact that following the optimal profile we will reach the value $1$ only at $+\infty$ thus a linear correction on a small set ensures that this value is achieved with a small cost in energy. 

\medskip\noindent\textit{Evaluation of $\F\e(\s\e,\phi\e)$:}
We prove inequality~\eqref{eq:limsup} for the sequence we have produced. Since the sets $I_{a\e}$, $I_{b\e}$, $ I_{c\e}$ and $I_{d\e}$ are disjoint we can split the energy as follows 
\begin{equation}\label{dislimsup0}
\F\e(\s_\eps,\phi\e)=\F\e(\s_\eps,\phi\e;I_{a\e})+\F\e(\s_\eps,\phi\e;I_{b\e})+\F\e(\s_\eps,\phi\e;I_{c\e})+\F\e(\s_\eps,\phi\e;I_{d\e})
\end{equation}
and evaluate each component individually. Since $\s\e$ is null and $\phi\e$ is constant and equal to $1$ in $I_{d\e}$ we have that $\F\e(\s,\phi\e;I_{d\e})=0$. For the other components we strongly use the definitions in~\eqref{def:se} and~\eqref{def:phie}. Firstly we split again the energy on the set $I_{a\e}$ as following   
\begin{equation*}
 \F\e(\s\e,\phi\e;I_{a\e})=\F\e(\s\e,\phi\e;R\e)+\F\e(\s\e,\phi\e;Q_{r\e}(0,0))+\F\e(\s\e,\phi\e;Q_{r\e}(l,0)). 
\end{equation*}
Now identity~\eqref{eq:energiapunti} leads to the estimate
\begin{equation*}
 \F\e(\s\e,\phi\e;Q_{r\e}(0,0))=\F\e(\s\e,\phi\e;Q_{r\e}(l,0))=\frac{\eta^2}{2\eps}c(\a,\t)+\frac{(1-\eta)^2}{2\eps}\;r\e^2
\end{equation*}
and 
\begin{equation*}
 \F\e(\s\e,\phi\e;R\e)=\lt[\frac{1}{2\eps}\eta^2 \lt|\frac{\t}{2a\e}\rt|^2+\frac{(1-\eta)^2}{2\eps}\rt]|R\e|\leq \lt[\frac{(\t\eta)^2}{8\eps a\e^2}+\frac{1}{2\eps}\rt]2a\e l.
\end{equation*}
Then passing to the limsup we obtain
\begin{equation}\label{dislimsup1}
 \limsup_{\eps\dw0}\F\e(\s\e,\phi\e;I_{a\e})\leq  \t\a l= \t_{}\a\H^1([0,l]\times\{0\}).
 \end{equation}
To obtain the inequality on the sets $I_{b\e}$ and $I_{c\e}$  we are going to apply the Coarea formula therefore let us observe that for both $d_\oo(x, I_{a\e})$ and $d_\oo(x, I_{b\e})$ there holds $|\nb d_\oo(x,\cdot)|=1$ for a.e. $x\in\Om$ and that there exist a constant $k=k(\a,\t)$ such that the level lines $\{d_\oo(x,\cdot)=t\}$ have $\H^1$ length controlled by $2l+kt$. In force of these remarks we obtain
\begin{align}\label{dislimsup2}
 \F\e(\s\e,\phi\e;I_{b\e})&=\int_{I_{b\e}} \lt[\frac{\eps}{2}|\nb \phi_\eps|^2  + \frac{(1-\phi_\eps)^2}{2\eps} \rt]|\nb d_\oo(x,I_{a\e})|\dx\nonumber\\
 &=\int_{0}^{b_\eps}\lt[\frac{(1-w_\eps(t))^2}{2\eps}+\frac{\eps}{2}|w'_\eps(t)|^2 \rt] \H^1(\{d_\infty(\cdot, I_{a\e})=t\})\;\dt\nonumber\\
 &\leq(2l+k\eps)\lt[\frac{1}{2}(1-w_\eps(t))^2\rt]_{0}^{b_\eps}\nonumber\\
 &=\lt(l-\frac{k\eps}{2}\rt)\lt[(1-\eta)^2-\eps^2\rt]\xrightarrow [\eps\dw0]{} l= \H^1([0,l]\times\{0\})
\end{align}
and 
\begin{align}\label{dislimsup3}
 \F\e(\s\e,\phi\e;I_{c\e})&=\int_{I_{c\e}} \lt[\frac{\eps}{2}|\nb \phi_\eps|^2  + \frac{(1-\phi_\eps)^2}{2\eps} \rt]|\nb d_\oo(x,I_{b\e})|\dx\nonumber\\
 &=\int_{0}^{\eps}\lt[\frac{(1-t+\eps-1)^2}{2\eps}+\frac{\eps}{2} \rt] \H^1(\{d_\infty(\cdot, I_{b\e}\cup I_{a\e})=t\})\;\dt\nonumber\\
 &\leq(2l+k\eps)\lt[\frac{(t-\eps)^3}{6\eps}+\frac{\eps}{2} t\rt]_{0}^{\eps}\nonumber\\
 &=(2l+k\eps)\,\eps^2\,\frac{2}{3}\xrightarrow [\eps\dw0]{} 0.
\end{align}
Finally adding up equations~\eqref{dislimsup0},~\eqref{dislimsup1},~\eqref{dislimsup2} and~\eqref{dislimsup3} we obtain
\begin{equation*}
\limsup_{\eps\dw0}\F\e(\s_\eps,\phi\e)\leq (1+\a\;\t)\; \H^1([0,l]\times\{0\}).
\end{equation*}

\medskip\noindent\textit{Case $\s$ of the form~\eqref{eq:strutturadense}:}

Let us call $\s\e^i$, $\phi\e^i$ the functions obtained above for each $\s_i=\t_i\xi_i \H^1\restr M_i$ and set
\begin{equation*}
\s\e=\sum_{i=1}^n \s\e^i,\qquad\qquad \phi\e=\min_{i}\;\phi\e^i.
\end{equation*}
Let us remark that in force of the local construction we have made at the ending points of each segment and since $\s$ satisfies equation~\eqref{eq:divconstraint1} for each $\eps$ there holds
\[\nb\cdot\s_\eps\, =\, \lt(N\delta_{x_0} - \sum_{j=1}^N \delta_{x_j} \rt)\ast \rho_\eps.\]
We now prove inequality~\eqref{eq:limsup}. The following inequality holds true
\begin{align}
\F\e(\s\e,\phi\e)&=\int_\Om \frac{1}{2\eps}|\min_i \phi\e^i|^2|\sum_{i=1}^n \s^i\e|^2 +\frac{\eps}{2}|\nb (\min_i \phi\e^i)|^2 + \frac{(1-\min_i \phi\e^i)^2}{2\eps}\dx\nonumber\\
&\leq\int_\Om \frac{1}{2\eps}|\min_i \phi\e^i|^2|\sum_{i=1}^n \s^i\e|^2 \dx+\sum_{i=1}^n\int_\Om\frac{\eps}{2}|\nb \phi\e^i|^2 + \frac{(1-\phi\e^i)^2}{2\eps}\dx,\label{dis:decomposizione}
\end{align}
therefore we look into an estimation of the first integral in the latter. Observe that for $\eps$ sufficiently small we can assume that all the $R\e^i$ are pairwise disjoint thus we study the behavior in the squares. Let  $M_{i_1},\dots,M_{i_{m_P}}$ be the segments meeting at a branching point $P$. For $j=i_1,\dots,i_{m_P}$ let us call $Q_{r\e^j}(P)$ the squared neighborhood of $P$ relative to the segment $M_j$ as constructed previously. Let us recall that by definition $\phi\e$ is constant and equal to $\eta$ on $\cup_{j=i_1}^{m_P} Q_{r\e^j}(P)$ then we have the estimation
\begin{align}
\int\limits_{\cup_{j=i_1}^{m_P} (R\e^j\cup Q_{r\e^j}(P))}\frac{\phi^2\e}{2\eps}\,|\s\e|^2\dx&=\sum_{j=i_1}^{m_P}\int_{R\e^j}\frac{\phi^2\e}{2\eps}\,|\s\e|^2\dx+\int_{\cup_{i=i_1}^{m_P} Q_{r\e^j}(P)}\frac{\phi^2\e}{2\eps}\,|\sum_{j=i_1}^{m_P}\s^j\e|^2\dx\nonumber\\
&\leq \sum_{j=i_1}^{m_P}\int_{R\e^j}\frac{\phi^2\e}{2\eps}\,|\s\e|^2\dx+m_P\,\frac{\eta^2}{2\eps}\sum_{j=i_1}^{m_P}\int_{Q_{r\e^j}(P)}|\s^j\e|^2\dx\nonumber\\
&\leq \sum_{j=i_1}^{m_P}\int\limits_{(R\e^j\cup Q_{r\e^j}(P))}\frac{1}{2\eps}|\phi\e^j|^2|\s^j\e|^2\dx+\underbrace{(m_P-1)\,\lt(\sum_{j=i_1}^{i_{m_P}}{c(\a,\t_j)}\rt)\frac{\eta^2}{2\eps}}_{c(m_P,\a,\t_{i_1},\dots,\t_{i_{m_P}})\eps}.\label{dis:punti}
\end{align}
Applying inequality~\eqref{dis:punti} on each branching point in equation~\eqref{dis:decomposizione} and recomposing the integral gives
\begin{align*}
\limsup_{\eps\dw0}\F\e(\s\e,\phi\e)&\leq\limsup_{\eps\dw0}\sum_{i=1}^n\F\e(\s^i\e,\phi^i\e) +n\; c(n,\a,\t_i,\dots,\t_n)\eps\\
&\leq \sum_{i=1}^n(1+\a\;\t_i)\; \H^1(M_i)\\
&=\int_{\supp(\s)}(1+\a\,\t)\dH^1 =\E_\a(\s,1)
\end{align*}
which ends the proof. 
\section{Numerical Approximation}\label{sec:NumApprox} 
\subsection{Equations}
In this section we present some numerical simulations of the $\Gamma$-convergence result we have shown. The first issue we address is how to impose the divergence constraint. To this aim is convenient to introduce the following notation 
\begin{align*}
 f\e&= \lt(N\delta_{x_0} - \sum_{j=1}^N \delta_{x_j} \rt)\ast \rho_\eps,\\
 G\e(\s,\phi)&=\begin{dcases}
                   \int_{\Om} \lt[\frac{1}{2\eps}|\phi|^2 |\s|^2 \rt]\dx& \text{if } \s\in V\e,\\
                   +\oo&\mbox{otherwise in } L^2(\Om,\R^2),
                 \end{dcases}\\
 \Lambda\e(\phi)&=\begin{dcases}
		\int_{\Om}\lt[ \frac{\eps}{2}|\nb \phi|^2 +\frac{(1-\phi^2)}{2\eps} \rt]\dx& \text{if } \phi\in W\e,\\
	      +\oo&\mbox{otherwise in } L^1(\Om).
		\end{dcases}
\end{align*}
Then let us observe that the following equality holds 
\[
\min_{\s\in L^2(\Om,\R^2)} G\e(\s,\phi)= \inf_{\s\in L^2(\Om,\R^2)}\lt\{ \sup_{u\in H^1(\Om)}\int_{\Om} \frac{1}{2\eps}|\phi|^2 |\s|^2 +u(\nb\cdot\s-f\e)\dx\rt\}.
\]
By Von Neumann's min-max Theorem~\cite[Thm.~9.7.1]{Att_Hed_Butt} we can exchange inf and sup obtaining for each $\eps>0$ and $\phi \in W\e$ 
\begin{align*}
\min_{\s} G\e(\s,\phi)&=\sup_{u}\inf_{\s}\int_{\Om} \frac{1}{2\eps}|\phi|^2 |\s|^2 -(\nb u\s+u f\e)\dx\\
&=-\min_{u}\int_{\Om} \frac{\eps|\nb u|^2}{2|\phi|^2}  +uf\e\dx=-\min_u G'\e(u,\phi).
\end{align*}
With the relation $\s=\frac{\eps \nb u}{\phi^2}$. This naturally leads to the following alternate minimization problem: given an initial guess $\phi_0$ we define
\begin{align*}
\s_j&:=\frac{\eps \nb u_j}{\phi_j^2} \quad\mbox{ where } \quad u_j:=\argmin G'\e(u,\phi_j), \\
\phi_{j+1}&:=\argmin G\e(\s_j,\phi)+\Lambda\e(\phi).
\end{align*}
We supplement the alternate minimization with a third step where we optimize the component $\Lambda\e$ with respect to a deformation of the domain. Let us describe this step. For $T:\Om \rw\Om$ a smooth map we define
\begin{equation}
\phi_T=\phi\circ T(x) \quad\mbox{and}\quad\Lambda\e(T)=\Lambda\e(\phi_T).
\end{equation}
By a change of variables we get
\[\Lambda\e(T)=\int_\Om\lt[\frac{\eps}{2}|(\nb T\circ T^{-1})\nb \phi|^2 +\frac{(1-\phi^2)}{2\eps}\rt]\det(\nb T^{-1})\dy\]
In particular we choose $T$ to be of the form $x+V(x)$ and evaluate the gradient obtaining 
\[\langle d\Lambda\e(T),W\rangle=\int_{\Om}\lt[\eps(\nb\phi_T;\nb W\nb\phi_T)-\frac{\eps}{2}|\nb\phi_T|^2\nb\cdot W-\frac{1}{2\eps}(1-\phi_T)^2\nb\cdot W\rt]\dx\]
Representing in $H^1(\Om,\Om)$ the gradient of the functional $\Lambda\e$ evaluated for $T(x)=x$ obtains the elliptic problem
\begin{equation*}
 \int_\Om\lt(\nb V,\nb W\rt)\dx+\int_{\Om}\lt[\eps(\nb\phi;\nb W\nb\phi)-\frac{\eps}{2}|\nb\phi|^2\nb\cdot W-\frac{1}{2\eps}(1-\phi)^2\nb\cdot W\rt]\dx=0
\end{equation*}
This method enhances the length minimization process since, as we already pointed out, $\Lambda\e$ is a variation of Modica-Mortola's functional.  
\subsection{Discretization}
We define a circular domain $\Om$ containing the points in $S$ endowed with a uniform mesh and four values $\a$, $\eps_{in}, \eps_{end}$ and $N_{iter}$ and a gaussian convolution kernel $\rho_{\eps_{end}}$ in order to define $f\e$. 
For the discrete spaces we have chosen for $u$, $\phi$ and the vector field $V$ to be piecewise polynomials of order $1$. This leads to the following algorithm

 \begin{algorithm}
 \caption{$\Gamma$-convergence} \label{alg:1}
 \begin{algorithmic}
 \Require 	
 	  $S=\{x_0,\ldots,x_N\}$, \quad$\eps_{in},\quad\eps_{end}$,\quad	 $N_{iter}$,\quad $\a$, index.
 \Function{Steiner}{$x_0,\ldots,x_N,\eps_{in},\eps_{end},N_{iter},\a,\rho$}
    \State  Set $f\e=(N\d_{x_0}-\sum_{i=1}^N \d_{x_i})*\rho_{\eps_{end}}$ and $\phi_0=1$
    \For{$j=1,\ldots,N_{iter}$}
    \State $\eps_j= \lt(\frac{j-N_{iter}}{N_{iter}}\rt)\eps_{in}-\lt(\frac{j}{N_{iter}}\rt)\eps_{end}$
    \State $\tilde\phi\leftarrow L^1$-projection of $\phi_{j-1}^2$
    \State Set $u_j$ as the minimizer of  $G'_{\eps_j}(\cdot,\phi_{j-1})$ 
    \State Set $\s_j=\frac{\eps_j \nb u_j}{\tilde\phi_{j-1}}$
    \State Set $\phi_j$ as the minimizer of $G_{\eps_j}(\s_j,\cdot)+\Lambda\e(\cdot)$
       \If {$j\%10==0 \;\& \;j \geq \; \mbox{index}$}
 \State Solve $\langle d\Lambda_{\eps_j}(T),W\rangle=0$
	\State Set $\phi_j=\phi_j(x+T)$
\EndIf
    \State Set $\phi_j=\max\{\eta,\phi_j\}$
\EndFor
 \EndFunction
 \State \Return  $\phi_{N_{iter}}$, $\s_{N_{iter}}$.
 \end{algorithmic}
 \end{algorithm}
We have implemented the algorithm in FREEFEM++. In the next figures we show the graphs obtained for the couple $(\s_{N_{iter}},\phi_{N_{iter}})$ via the approximation algorithm with the choices  $\a=0.05$, $\eps_{in}=0.5$, $\eps_{end}=0.05$, $\a=0.05$, $N_{iter} = 500$ and $index = 300$. 
We have chosen to make simulations for points located on the vertices of regular polygons of respectively 3, 4, 5 and 6 vertices. This choice allows a direct visual perception of the results.
\begin{figure}[h!]
\centering
  \includegraphics[width=3cm]{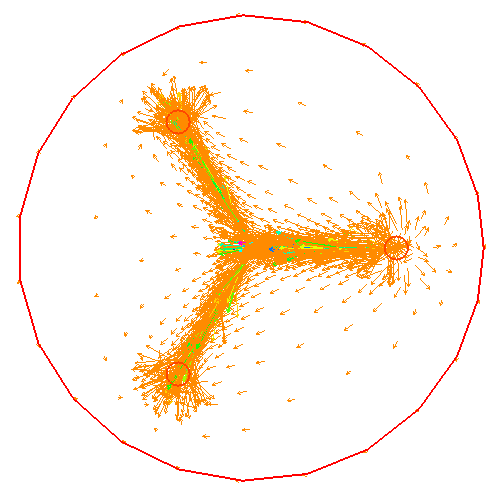}\quad \includegraphics[width=3cm]{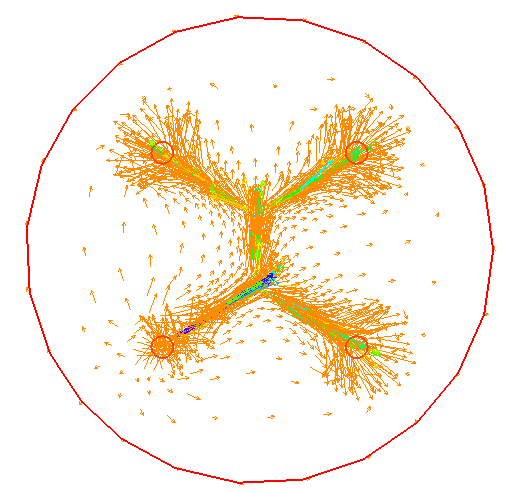}\quad\includegraphics[width=3cm]{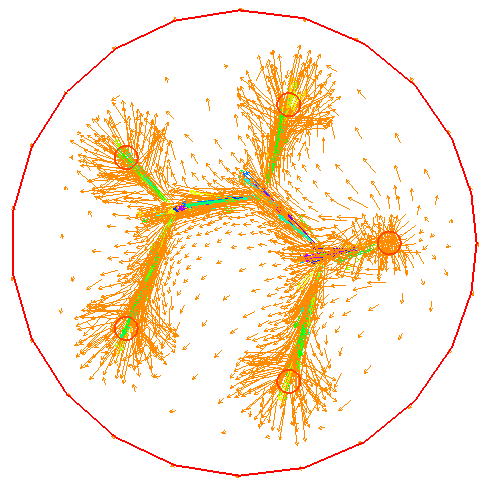}\quad \includegraphics[width=3cm]{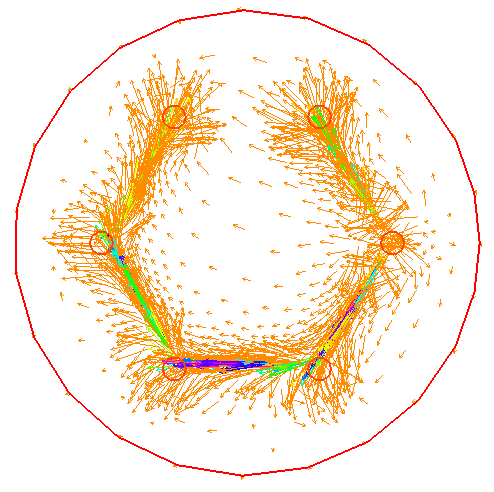}  \includegraphics[width=3cm]{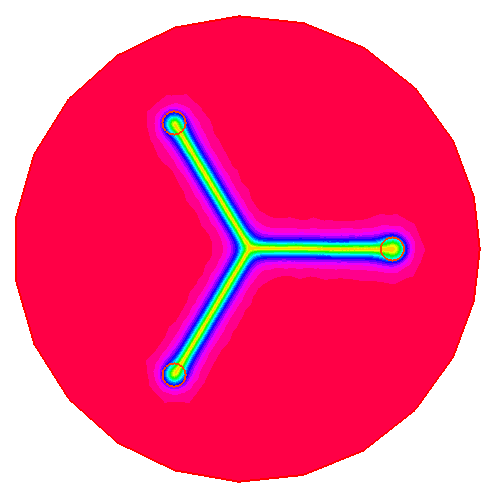}\quad \includegraphics[width=3cm]{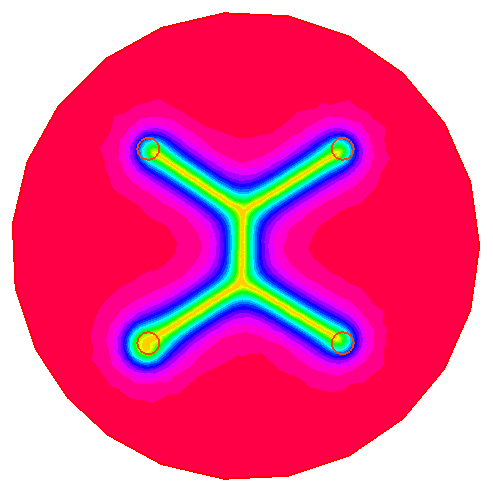}\quad\includegraphics[width=3cm]{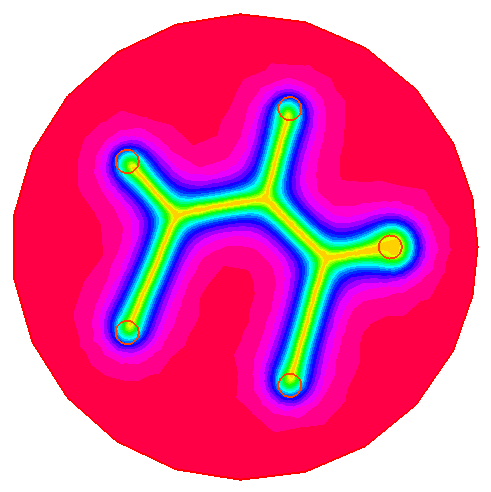}\quad \includegraphics[width=3cm]{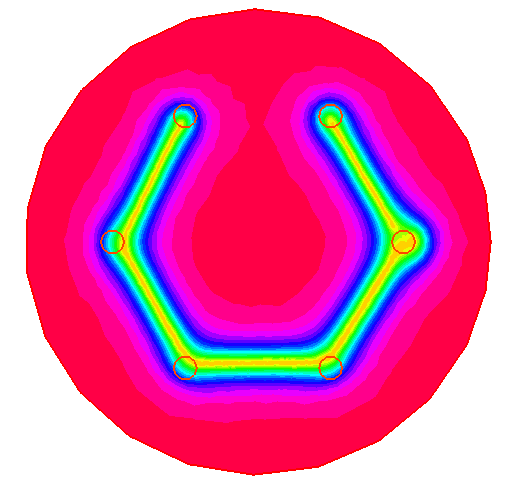}
  \caption{Graph of the couple $(\s_{N_{iter}},\phi_{N_{iter}})$ obtained via Algorithm~\ref{alg:1} in the case of 3, 4, 5 and 6 points located on the vertices of a regular polygon.}\label{fig:simulazione1}
\end{figure}
\begin{figure}[h!]
\centering
\includegraphics[clip,trim=0 2.5cm 0 2.5cm ,height=3.8cm]{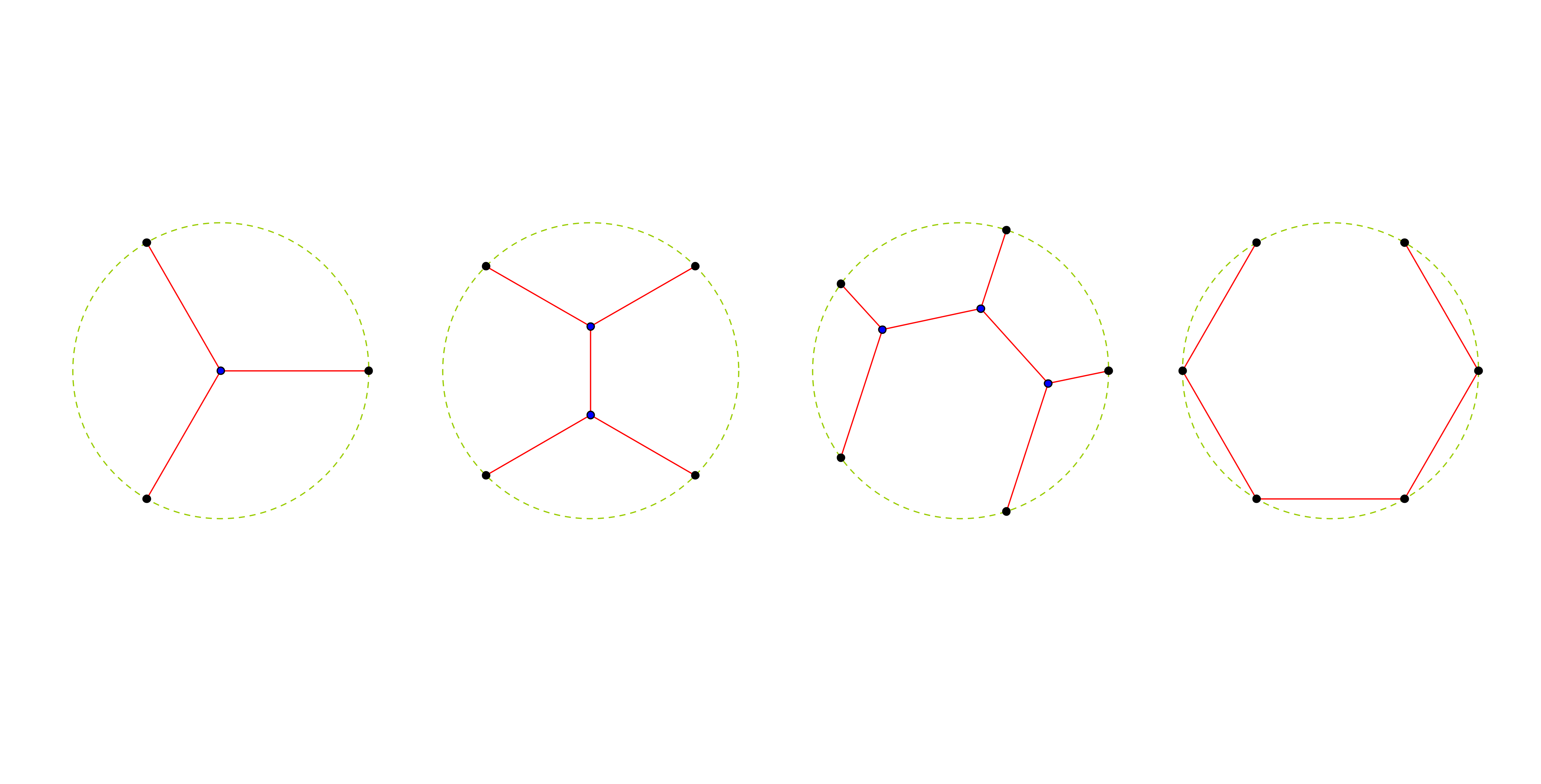}
  \caption{Graph of the exact solutions to the Steiner Problem constrained as the in the previous figure.}\label{fig:exact solution}
\end{figure}
Finally let us point out the need of the third minimization step. In the following figure we have the graph of the solution obtained for a simulation in which the third step is omitted. Even from visual perception is possible to recognize that the solution differs both from the solution of the Steiner Tree and the minimizer of the $\E_\a$ energy as evident from the figure. Furthermore we do not obtain the classical straight segments we would expect in studying geodesic in the euclidean metric. We suppose that these alterations are a consequence of the alternate minimization method that could not lead to a global minimum and therefore we introduced the third step in the algorithm in order to perturbate local solutions. 
 \begin{figure}[h!]
 \centering
 \includegraphics[width=3.5cm]{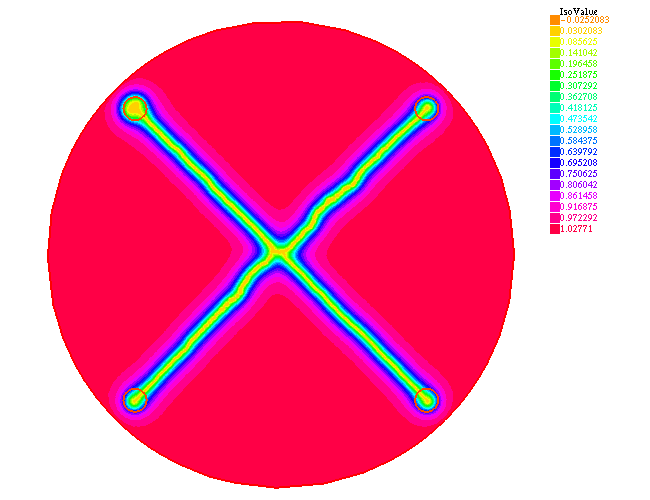}\qquad\qquad\qquad \includegraphics[width=3cm]{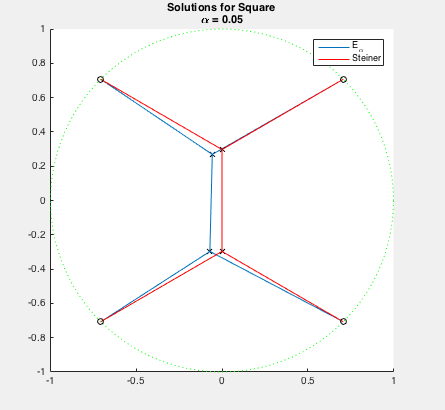}
  \caption{On the left: Graph of $\phi$ obtained via Algorithm~\ref{alg:1} in which the gradient descend method is omitted. On the right:  in red, one of the solutions to the Steiner problem for four points on the vertices of a square, while in blue, a minimizer of the energy $\E_\a$ associated to the same constraint.}\label{fig:4sbagliato}
 \end{figure}
 \newpage
\noindent To ensure that this step is reasonable we have studied several experiments and plotted the numerical energy of each experiment and observed that we are always led to a lower energy. The following plot shows the behavior of the energy for the iterations concerning the third step for the first two solutions in figure~\ref{fig:simulazione1}. Is possible to observe that although there are increments 
\begin{figure}[h!]
	\centering
	\includegraphics[width=.6\textwidth]{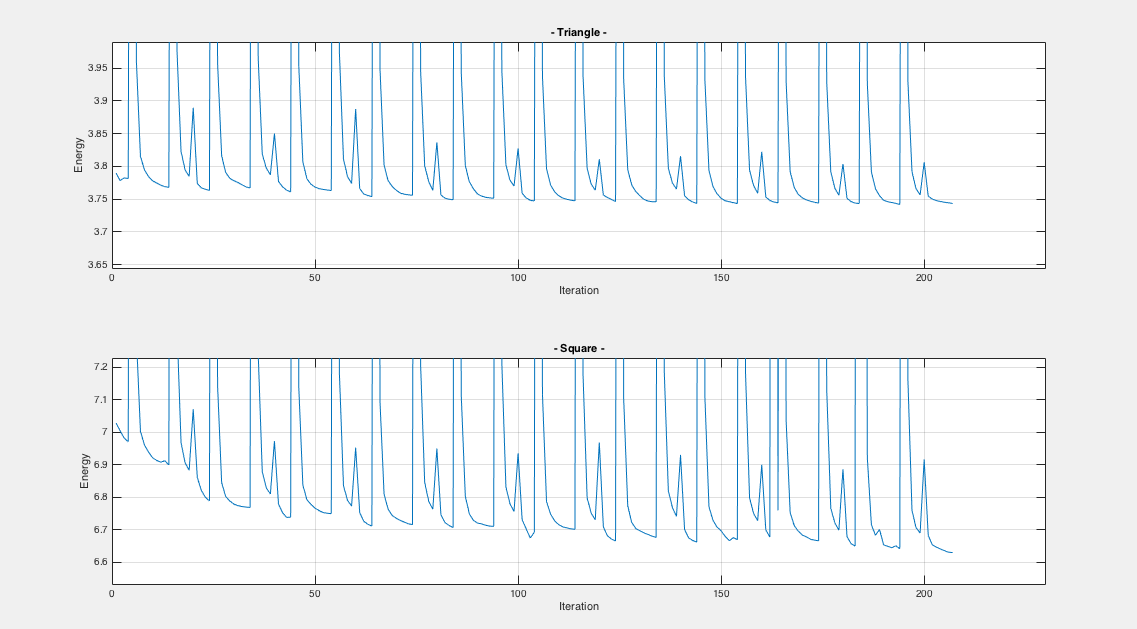}
	\caption{Behavior of the estimated energy of the last 200 iterations of Algorithm~\ref{alg:1} referring to the first two figures in figure~\ref{fig:simulazione1}
	.}\label{fig:grafici_energia}
\end{figure}

\section*{Acknowledgments}
The authors have been supported by the ANR project Geometrya, Grant No. ANR-12-BS01-0014-01. A.C.~also acknowledges the hospitality of Churchill College and DAMTP, U.~Cambridge, with a support of the French Embassy in the UK, and a support of the
Cantab Capital Institute for Mathematics of Information.

\bibliographystyle{plain}
\bibliography{bib}

 \end{document}